\documentclass{amsart}
\usepackage{amssymb,stmaryrd}
\usepackage{amsfonts}
\usepackage{amstext}
\usepackage{algorithmic}
\usepackage{algorithm}
\usepackage{young}
\usepackage[vcentermath]{youngtab}
\usepackage{ytableau}
\usepackage[pdftex]{graphicx}
\usepackage[all]{xy}
\usepackage{enumerate}
\usepackage{mathdots}

\usepackage{color}

\numberwithin{equation}{section}
\newtheorem{theorem}{Theorem}[section]
\newtheorem{lemma}[theorem]{Lemma}
\newtheorem{proposition}[theorem]{Proposition}
\newtheorem{corollary}[theorem]{Corollary}

\theoremstyle{definition}
\newtheorem{definition}[theorem]{Definition}
\newtheorem{example}[theorem]{Example}
\newtheorem{conjecture}[theorem]{Conjecture}

\theoremstyle{remark}
\newtheorem{remark}[theorem]{\bf{Remark}}

\newcommand{\R}{{\mathbb{R}}}

\newcommand{\C}{{\mathbb{C}}}
\newcommand{\Z}{{\mathbb{Z}}}
\newcommand{\Q}{{\mathbb{Q}}}

\newcommand{\N}{{\mathbb{N}}}

\newcommand{\<}{{\langle}}
\renewcommand{\>}{{\rangle}}
\newcommand{\Tr}{{\rm Trace}}
\newcommand{\cg}{{\mathfrak{g}}}

\newcommand{\CC}{{\mathcal{C}}}
\newcommand{\CL}{{\mathcal{L}}}

\renewcommand{\ker}{{\rm{ker}}}

\newcommand{\tens}{\otimes}
\newcommand{\id}{\rm id}

\newcommand{\extd}{{\rm d}}

\newcommand{\eps}{\epsilon}
\newcommand{\ev}{{\rm ev}}
\newcommand{\coev}{{\rm coev}}
\newcommand{\und}{\underline}

\newcommand{\CGe}{{\CC\subseteq G\setminus\{e\}}}

\begin{document}

\title[Lie theory of finite simple groups and Roth property]{Lie
  theory of finite simple groups and the Roth property}
\keywords{Finite group, Lie algebra, Riemannian geometry, conjugacy
  class, noncommutative geometry, quantum group}
\subjclass[2000]{Primary 81R50, 58B32, 20D05}

\author{J L\'opez Pe\~na}
\address{University College London\\
  Department of Mathematics, Gower Street, London WC1E 6BT, UK}
\email{jlp@math.ucl.ac.uk}

\author{S Majid}
\address{Queen Mary University of London\\
  School of Mathematical Sciences, Mile End Rd, London E1 4NS, UK}
\email{s.majid@qmul.ac.uk}

\author{K Rietsch} \address{Kings College London\\ Department of
  Mathematics, The Strand, London, UK}
\email{konstanze.rietsch@kcl.ac.uk} \thanks{The
  1st author was funded by an EU Marie-Curie fellowship
  PIEF-GA-2008-221519 and MCIM grant MTM2010-20940-C02-01, the 2nd
  author by a Leverhulme research fellowship and the 3rd author by an
  EPSRC research fellowship EP/D071305/1}


\begin{abstract}
  In noncommutative geometry a `Lie algebra' or bidirectional bicovariant
  differential calculus on a finite group is provided by a choice of an ad-stable generating
  subset $\CC$ stable under inversion. We study the associated Killing form.
  For the universal calculus associated to $\CC=G\setminus\{e\}$ we show that the magnitude
 of the Killing form $\mu=\sum_{a,b\in\CC}K^{-1}_{a,b}$ is defined for all finite groups (even when $K$ is
 not invertible) and that a finite group is Roth, meaning its conjugation representation contains every irreducible,  {\em iff} $\mu\ne \frac{1}{N-1}$ where $N$ is the number of conjugacy classes. We show further that the Killing form is invertible in the Roth case, and that the Killing form restricted to the $(N-1)$-dimensional subspace of invariant vectors is invertible {\em iff} the finite group is almost-Roth (meaning its conjugation representation has at most one missing irreducible). It is known \cite{HZ,Tiep} that most nonabelian finite simple groups are Roth and that
 all are almost Roth.  At the other extreme from the universal calculus we prove that the 
 the generating conjugacy class in the case of  the dihedral groups
  $D_{2n}$ with $n$ odd has invertible Killing form, and the same for the
  $2$-cycles conjugacy class in any $S_n$. We also compute some eigenvalues of the
  Killing form in the case of the $n$-cycles class in $S_n$. Finally, we verify invertibility of the Killing forms of all 
  real conjugacy classes in all nonabelian finite simple  groups to order
  $75,000$, by computer, and we conjecture this to extend to all nonabelian finite simple groups. 
\end{abstract}

\maketitle

\section{Introduction}

In this paper we demonstrate the existence of a useful `Lie theory' of
finite groups with a detailed study of the Killing form. We recall that historically
the theory of Hopf algebras has unified enveloping algebras of Lie algebras 
with group algebras. In the same way there seems to be a reasonable
generalisation of Lie algebras themselves, which comes out of quantum groups and their
noncommutative differential geometry and which can nevertheless be
specialised to finite groups. Here the `Lie problem' of finding a
finite-dimensional Lie algebra-type object associated to the
Drinfeld-Jimbo quantum groups $U_q(\cg)$ was solved in \cite{Ma:blie}
in the form of a `braided-Lie algebra', consisting of a coalgebra
$\CL$ in a braided category and a bracket operation $[\ ,\ ]:\CL\tens\CL\to\CL$ subject to
certain axioms. This will be recalled in the preliminary Section~2 where
we will cover the reduction to the finite group case. For a braided Lie algebra 
there is also a notion of braided Killing form $K:\CL\tens \CL\to \und 1$ defined as a
braided-trace of $[\ ,\ ](\id\tens [\ ,\ ])$. When the category in
which  $\CL$ lies is 
Abelian, one has a quadratic braided enveloping algebra
$U(\CL)$ which forms a bialgebra and which in some cases can quotient to a
Hopf algebra in the category. Within this framework, for $U_q(\cg)$ and at least for
generic $q$ we have a certain $\CL\subset U_q(\cg)$ and
$U(\CL)\twoheadrightarrow B_q(G)$, where the latter is a braided
version of the quantum group of which $U_q(\cg)$ is a localisation
(alternatively one can work over formal power-series). In this
context, where $\cg$ is semisimple, the braided-Killing form is
nondegenerate as an expression of the factorisability of the quantum
group cf. \cite{Ma:blie,Ma:rieq}.  Just as in Lie theory, the
braided-Lie algebras here arise \cite{GomMa} from bicovariant
differential structures on quantum groups \cite{Woronowicz} but the
usual theorem that a topological group has at most one differentiable
structure making it a Lie group does not apply and rather there is a
known classification theory for the differential structures and hence
of braided-Lie algebras, for each $U_q(\cg)$. Also the usual theorem that a discrete
topology admits only the 0--dimensional differential structure does
not apply and this means that we can specialise to finite groups. Many
differential constructions still work and in particular one has a
notion of noncommutative de~Rham cohomology for each choice of
calculus.

We will not need the full extent of this theory, being interested in
the case where the category is that of vector spaces over a field $k$
with the trivial `flip' braiding and trivial associator. The general
framework, however, provides a bridge
\begin{equation} \label{functor}
	\begin{array}{rcl} 
		& {\rm Quantum\ Groups} & \\
		\swarrow & & \searrow \\ 
		{\rm Lie\ Algebras} & & {\rm Finite\ Groups}
	\end{array}
\end{equation}
for the transfer of ideas from Lie theory to finite groups (taking
ideas backwards up the left arrow is a loosely defined process of
`quantisation' and we then specialise down by the right arrow).  

For
the specialisation of structure represented by the first arrow one can
look at braided Lie algebras of the form $\CL=k\oplus\cg$, a linear
map $[\ ,\ ]:\cg\tens\cg\to \cg$ and a specific form for the remaining
structure (see Section~2). The axioms of a braided-Lie algebra then
reduce to those of a Leibniz algebra on $\cg$, which includes the case
of an ordinary Lie algebra. In the Lie case we obtain 
$U(\CL)\twoheadrightarrow U(\cg)$ as a quadratic bialgebra extension
of the usual enveloping algebra. The braided-Killing form extends the
usual Killing form, and is nondegenerate if and only if the usual
Killing form is.  Other choices of $\CL$ can be found from the degree
filtration of $U(\cg)$. 

On the right hand side we can consider braided-Lie
algebras of the form $\CL=k\CC$ where $\CC$ is a set and we take the
diagonal coalgebra structure. Then the axioms of a regular braided-Lie
algebra reduce to a set map $[\ ,\ ]:\CC\times\CC\to \CC$ obeying the
axioms of a left-handed rack.  A quandle is a rack with a further
restriction (see Section~2.1) and arises naturally when $\CC$ is an ad-stable
generating subset $\CC\subseteq G\setminus\{e\}$ in a finite group $G$.
Here $e$ is the group identity. We
consider such a quandle $\CC$ as playing the role of a `Lie algebra' for $G$.  
In this setting we also have a quadratic bialgebra $U(\CL)\twoheadrightarrow kG$ and a Killing form
which looks like 
\begin{equation}\label{killing}
	K(a,b)=|Z(ab)\cap\CC|,\quad \forall a,b\in\CC,
\end{equation}
where $Z(g)$ is the centraliser of $g\in G$. Note that $K(a,b)$ is the 
trace of $ab$ in the conjugation representation of $G$ on
$\CL=k\CC$.  We also
associate some constants to the Killing form when it is sufficiently invertible,
the most important of which is
\[
\mu = \sum_{a,b\in \CC} K^{-1}_{a,b}
\]
the sum of all matrix entries of $K^{-1}$ in our basis $\CC$. This is the magnitude of 
the matrix $K$ in the sense of \cite{Lei} and can be defined even when $K$ is not invertible
as long as the vector with all entries 1 is in the image. 

In this paper we work over $\C$ and study the Killing form (\ref{killing}), particularly the following question 
motivated by the above transfer of ideas from Lie theory: just as a
Lie algebra over $\C$ is semisimple if and only if the Killing form is
nondegenerate, {\em is the Killing form for an ad-stable inverse
  stable generating subset $\CC$ (as `Lie algebra') nondegenerate when the group
  is simple or a product of simple groups?} This is a bit too much to
  ask in general, as we will see, so we make the following definition.

\begin{definition}\label{nondeg} 
  Let $G$ be a finite group. If the Killing form is nondegenerate
  \begin{enumerate}
  \item for every ad-stable inversion-stable generating subset
    $\CC\subseteq G\setminus\{e\}$, we say that $G$ is {\em strongly
      nondegenerate}
  \item for the universal calculus $\CC=G\setminus\{e\}$, we say that
    $G$ is {\em nondegenerate}
  \item for every nontrivial real generating conjguacy class, we say
    that $G$ is {\em class-nondegenerate}
  \end{enumerate}
\end{definition}

Clearly (1) is the most desirable in the absence of a particular
choice of $\CC$ as `Lie algebra'. (2) is not very classical (the
universal calculus is very far from the classical one on a Lie group,
and has the undesirable property of yielding a trivial de Rham
cohomology) but is the simplest to look at and we will achieve a more
or less complete analysis of this case. (3) is reasonable if we think
that a `Lie algebra' should be in some sense minimal. It is also a
proxy for (1) since any $\CC$ is a disjoint union of conjugacy classes
and one might expect that if (3) holds then (1) will tend to hold as
well. Although it is not part of the classical analogy we also say
that a group is {\em absolutely nondegenerate} if (1) holds for all
$\CC$ not only generating ones. There is no difference with (1) in the
simple case.

Our main results concern $G$ a finite group with what we call the {\em
  Roth property}  that every irrep is contained in the conjugation
representation. We prove (Theorem~\ref{roth}) that every Roth property
group is nondegenerate. Moreover, we show (Theorem~\ref{rothmu}) that the magnitude $\mu$ of the
Killing form for the universal calculus is defined for every finite group and completely
characterises the Roth property. Namely,  we show that a finite group is Roth iff 
 $\mu\ne \frac 1{N-1}$ where $N$ is the number of conjugacy classes in the group, and
 in this case we give a formula for $\mu$. 
 The background here is that Roth's conjecture
\cite{Roth} in the theory of finite groups asserted that the conjugation
representation of the group $G$ on the vector space $\C[G]$ contains 
every complex irreducible representation of $G/Z(G)$
at least once (where $Z(G)$ is the centre of $G$). Roth's conjecture turned
out to be false in general, but is known to be true for symmetric
groups \cite{Frumkin} and alternating groups \cite{Scharf}, and,
recently, for the sporadic simple groups \cite{HZ} using methods from
\cite{Pass}.  Indeed, for simple nonabelian groups the exceptions
 amount to some instances of one
classical family of Lie type over finite fields of particular order\cite{HZ} and in these cases the conjugation representation lacks 
exactly one irrep\cite{Tiep}. In this case
we do not necessarily have nondegeneracy but a weaker result applies that Killing form is nondegenerate when
restricted to the subspace of invariant vectors (Proposition~\ref{deg}).
We also show there that if a finite
group lacks two or more irreps in its conjugation representation
then it is not nondegenerate. Hence if a finite group is nondegenerate
and not Roth then it must indeed lack precisely one irrep in its conjugation
representation. The non-degeneracy property in our Killing form approach 
thus provides a new point
of view on the Roth property, and may even coincide with it for finite simple
groups. The two properties are not equivalent in general however, 
see below. Meanwhile, the magnitude of the Killing form provides a complete
characterisation of whether a finite group is Roth or not.

Let us write $\Z_k$ for the cyclic group $\Z/k \Z$.
We first show in Corollary~\ref{ZG} that any nondegenerate group of
order $|G|>2$ is centreless while the group $\Z_2$ is exceptional in
being nondegenerate and not centreless. Hence we have the following
picture
\[ {\rm Most\  Nonabelian\ Simple}\subset {\rm Roth\ property}\subset
{\rm Nondegenerate}\subset {\rm Centreless}\cup\{\Z_2\}
\]
where on the left we mean all simple nonabelian groups including sporadics
with the possible exceptions identified in \cite{HZ}.  All inclusions
here are strict. For instance $S_n$, $n\ge 3$ and $D_{2n}$, odd $n\ge
3$ are Roth but of course they are not simple. Meanwhile the group
\[
(((\Z_5\times \Z_5) \rtimes \Z_4) \rtimes \Z_2) \rtimes \Z_2
\]
of order 400 (labeled (400, 207) in the Small Groups Library
\cite{SG}) is centreless and nondegenerate but not Roth (indeed we find that it is
the smallest such example). The last inclusion is also strict as many
centreless groups are not nondegenerate. Of the 680 centreless groups
of order $|G|\le 500$ some 537 are nondegenerate. These results were
found using GAP \cite{GAP} and Sage \cite{Sage}.

The above results all pertain to the universal calculus and its
corresponding class of nondegenerate groups. 
At the other extreme we have the 
calculi associated to real, generating conjugacy classes, and the property of class non-degeneracy.
All 680 centreless groups of order $|G|\le 500$ are class-nondegenerate, although this includes
452 of them which do not actually have any real generating conjugacy classes.
In this context we make the specific conjecture:
\begin{conjecture}\label{con}
  All nonabelian finite simple  groups are class-nondegenerate.
\end{conjecture}
This is supported by computer analysis where we have checked this conjecture for
all $|G|\le 75,000$.  This was again done using Sage and the methods and
tables are collected in the Appendix. Thus
\[
\begin{array}{rcl}
{\rm Class\ Nondegenerate} & \supset {\rm Strongly\ Nondegenerate}\subset &{\rm Nondegenerate} \\
\\
\cup_{\rm conjecture}&\cup_{\rm many\ but\ not\ all} & \quad \cup_{\rm most\ but\ not\ all} \\
\\
& {\rm Finite\ Simple\  Nonabelian\ Groups,} &
\end{array}
\]
where Conjecture~\ref{con} is that all finite simple nonabelian
groups are included on the left and Theorem~\ref{roth} combined with
\cite{HZ} says that most  finite simple nonabelian groups
are included on the right. Up to order 75,000 the first non-Roth finite simple nonabelian group
is $PSU(3,3)$ according to \cite{Tiep} and one can check that this is {\em not} nondegenerate.  The only other
non-Roth finite simple nonabelian group to this order is $PSU(3,4)$ which at
order approximately $62,000$ is well beyond direct verification.  With possibly a small number of exceptions it would appear then
that most finite simple groups are nondegenerate at the two extremes
and indeed that many of them are strongly nondegenerate. Using Sage, we 
find that of the 15
simple nonabelian groups to order $|G|\le 8000$ in the tables in the
Appendix, the groups $A_5$, $A_6$, $A_7$, $PSL(2,8)$, $PSL(2,13)$,
$PSL(2,17)$, $PSL(3,3)$, $PSL(2,16)$, $PSL(2,25)$, and the sporadic Mathieu group $M_{11}$ {\em are}
strongly nondegenerate, whilst $PSL(2,7)$, $PSL(2,11)$, $PSL(2,19)$,
$PSU(3,3)$ and $PSL(2,23)$ are not strongly nondegenerate
(these are still class-nondegenerate and all but $PSU(3,3)$ are nondegenerate). This explains the above picture and somewhat
answers our original question to the extent currently within reach.

We will also prove in Corollary~\ref{Sn2cy} that the Killing form for
$S_n$ ($n\ge 3$) in the case where $\CC$ is taken to be the set of 2-cycles is
nondegenerate. This, combined with the Roth property and other data,
suggests that $S_n$ for $n\ge 5$ is at least class-nondegenerate. 
We have verified that  in fact it is absolutely non-degenerate up to and including $S_8$. 
Likewise, $D_{2n}$ for $n$ odd is
nondegenerate, class-nondegenerate (see Proposition~\ref{D2n}) and
is possibly strongly nondegenerate, but it not  absolutely
nondegenerate. We have checked strong nondegeneracy for odd $n$ dihedral 
groups $D_{2n}$ up to order $50$. Hence there are plenty of 
groups which are nonsimple but strongly
nondegenerate. We also mention $S_3,S_4,A_4$ as some other examples of
groups which are nondegenerate, strongly nondegenerate but not
absolutely nondegenerate.

Finally, the product of two Roth groups is clearly Roth, and we see
(Proposition~\ref{produni}) in our Killing form approach that this
characterises Roth groups among all nondegenerate ones. 
At the other extreme, suppose two groups $G_1$ and $G_2$ have  
`Lie algebras' with non-degenerate Killing form coming from  
conjugacy classes $\CC_1,\CC_2$. Then in Proposition~\ref{Kdis} 
we characterize, when the direct product $G_2\times G_2$
has nondegenerate Killing form on its `Lie algebra' obtained by forming the disjoint union 
$\CC_1\sqcup\CC_2$ in $G_1\times G_2$. 


Among further results we show that when nondegenerate, the matrix $K$ is
positive definite precisely when $\CC$ has consists only of elements
of order $2$. More generally the index in the sense of positive minus
negative eigenvalues is the number of elements of order $2$, see
Proposition~\ref{sig}. By the Feit-Thompson theorem every finite
simple group has at least one element of order $2$. The 
corresponding conjugacy class therefore gives us a choice of $\CC$ for which
the braided-Killing form is positive definite if it is
nondegenerate. This is a little reminiscent of usual Lie theory where
a complex simple Lie algebra has a compact real form where the Killing
form is negative definite.

The Killing form appears to have further properties
that are suggested by our data but which are not understood precisely
enough to formulate as a conjecture. The most important such observation
is that the Killing form decomposition of $\C\CC$ into eigenspaces tends to
be a decomposition into irreducibles or conjugate pairs of them.  We illustrate
this for $S_n$ and the 2-cycles class in Section~6. Finally we  compute 
some eigenvalues of the Killing form on the 
$n$-cycles class using a formula of
Zagier's and a related conjecture. 

Although the general picture makes sense over any field, all sections
after Section~2 will be over $\C$ (or a suitable splitting field for
the relevant groups).

\section{From braided-Lie algebras to `Lie algebras' on finite groups}

In this section we make precise (\ref{functor}) and thereby provide
the context of braided-Lie algebras which underpins the point of view
in the rest of the paper. We derive our point of view of an ad-stable
generating subset $\CC\subseteq G\setminus\{e\}$ as a `Lie algebra'
and the Killing form in Example~\ref{quand} which is then used in the
rest of the paper.

We recall that a braided category means a monoidal category $\tens$
with unit object $\und 1$ and natural isomorphisms
$\Psi:\tens\to\tens^{op}$, $\Phi: \tens(\ \tens\ )\to (\ \tens\
)\tens$ subject to standard triangle, hexagon and pentagon
identities. The associator $\Phi$ can be omitted since by Mac Lane's
theorem it can be inserted as needed for brackets to make sense, while
the braiding $\Psi$ is denoted by a crossing in a diagrammatic
notation in which these and other morphisms are read flowing down the
page and $\tens$ is denoted by juxtaposition \cite{Ma:prim}. Algebra
in such a category is done as `flow charts' except that under and over
crossings are significant.

A braided-Lie algebra \cite{Ma:blie} is a quadruple
$(\CL,\Delta,\eps,[\ ,\ ])$ where $\CL$ is an object on a braided
category, $\Delta:\CL\to \CL\tens\CL$ and $\eps:\CL\to \und 1$ makes
it a coalgebra (the axioms are those of a unital algebra but with
arrows reversed), and $[\ ,\ ]:\CL\tens\CL\to \CL$ is a braided
coalgebra map obeying the axioms 
\[ [\ ,\ ]([\ .\ ]\tens [\ ,\ ])(\id\tens\Psi\tens\id)(\Delta\tens\id\tens\id)=[\ ,\ ](\id\tens[\ ,\ ])\]
\[ (\id\tens[\ ,\ ])(\Psi\tens\id)(\id\tens\Psi^2)(\Delta\tens\id)=(\id\tens[\ ,\ ])(\Delta\tens\id).\]
Here a braided
coalgebra map means we are considering the braided tensor product
coalgebra structure on $\CL\tens\CL$. Since we are only interested in
tensor products and sums of one object one can also think of a single
braided-Lie algebra as a sextuple $(\CL,\Delta,\eps,[\ ,\
],\Psi,\Phi)$ where $\Psi:\CL\tens\CL\to \CL\tens \CL$ and
$\Phi:\CL\tens(\CL\tens\CL)\to (\CL\tens\CL)\tens\CL$ and subject to
similar axioms.
 
\begin{lemma}\cite{Ma:sol} 
  For any braided-Lie algebra $\CL$ the morphism
  $\tilde\Psi:\CL\tens\CL\to \CL\tens\CL$ defined by
  \[
 \tilde\Psi=([\ ,\ ]\tens\id)(\id\tens\Psi)(\Delta\tens\id)
  \]
  obeys the braid relations on $\CL^{\tens 3}$. We call the braided
  Lie algebra {\em regular} if $\tilde \Psi$ is invertible.
\end{lemma}

Associated to any braided Lie algebra $\CL$ in an abelian braided
category there is a quadratic bialgebra $U(\CL)$ in the braided category.
It is defined as the tensor algebra $T\CL$ modulo the relations given by
coequalizing the multiplication maps $\mu$ and $\mu \circ \tilde{\Psi}$,
and with coalgebra structure defined by extending that of $\CL$.

Associated to any braided-Lie algebra with (say) a left dual in the
braided category (a rigid object), there is a notion of
`braided-Killing form' $K:\CL\tens\CL\to \und 1$ defined as
\[ K=\ev\Psi( [\ ,\ ]\tens\id)(\id\tens[\ ,\ ]\tens\id) (\id\tens\id\tens\coev)\]
where $\ev:\CL^*\tens\CL\to \und 1$ and $\coev:\und 1\to \CL\tens\CL^*$ are the
evaluation and coevaluation and here $\Psi:\CL\tens\CL^*\to \CL^*\tens\CL$. This has the
form of a braided trace of $[\ ,[\ ,\ ]]$. Key properties including
invariance under the action of $[\ ,\ ]$ and braided-symmetry
$K=K\tilde\Psi$ are shown in \cite{Ma:blie}. The Killing form $K$ is
invertible in a standard categorical sense if there is another
morphism $K^{-1}:\und 1\to \CL\tens\CL$ such that
\[
(\id\tens K)(K^{-1}\tens\id)=\id = (K\tens\id)(\id\tens K^{-1})
\]
where these are morphisms $\CL\otimes\CL\otimes \CL\to\CL \otimes\CL\otimes \CL$.

\begin{proposition}\cite{Ma:blie}\label{quand2} 
  If $K$ is invertible then
  \[
  [\ ,\ ](\id\tens\sigma)\Delta=\id
  \]
  where $\sigma:\CL\to \CL$ is $\sigma = (\id\tens K) (\Psi\tens\id)
  (\id\tens K^{-1})$
\end{proposition}

\begin{example}
  An actual Lie algebra $\cg$ can be seen as a braided-Lie algebra of
  the form $\CL=k\oplus \cg$ in the category $\rm Vec$ of vector
  spaces over $k$ (so with trivial braiding of the underlying
  category, although a nontrivial $\tilde\Psi$ even in this case,
  provided that the Lie bracket is nonzero). Here
  \[
  [c,v] = v,\quad [v,c] = 0,\quad [c,c] = c,\quad \Delta v = v\tens
  c+c\tens v,\quad \Delta c = c\tens c,\quad\forall v\in\cg
  \]
  where $c$ spans the copy of $k$. The axioms of a braided-Lie algebra
  then amount to the bracket $[\ ,\ ]:\cg\tens\cg\to \cg$ obeying
  \[
  [[v,w],z]+[w,[v,z]]=[v,[w,z]],\quad\forall v,w,z\in\cg
  \]
  while regularity is automatic as is the property $[\ ,\
  ]\Delta=\id$. We do not require antisymmetry of the bracket which
  means that a braided-Lie algebra of this form is the same as saying
  that $\cg$ is a Leibniz algebra, a slightly more general notion than
  that of a Lie algebra but including it. Here $U(k\oplus\cg)$ is a
  quadratic bialgebra associated to any Leibniz algebra with relations
  $xy-yx=c[x,y]$ and $c$ central. In the Lie algebra case there is a
  bialgebra homomorphism $U(k\oplus\cg)\to U(\cg)$ sending $c=1$. The
  Killing form restricts to the usual Killing form and in addition
  \[
  K(c,c)=1,\quad K(c,x)=K(x,c)=0.
  \]
\end{example}

\begin{example}\label{quand}
  Similarly, we can consider $\CL=k\CC$ where $\CC$ is a set, and
  $\Delta a=a\tens a$ and $\eps(a)=1$ for all $a\in \CC$. Writing
  $[a,b]={}^ab$ as a notation, the axioms boil down in this case to
  \[
  {}^{({}^ab)}({}^a c)= {}^a({}^bc),\quad\forall a,b,c\in \CC.
  \]
  The regularity condition amounts to the requirement that for every
  $a,c$ there is a unique $b$ such that ${}^ab=c$.  Such a structure
  is variously called a `rack'. A `quandle' as opposed to a rack has
  the further condition ${}^aa=a$ and this is expressed in braided-Lie
  algebra terms as the further condition $ [\ ,\ ]\Delta=\id$. We
  assume henceforth that $\CC$ is finite. Then the Killing form on
  basis elements is clearly
  \[ K(a,b)=\Tr_\CL{}^a({}^b(\ ))=|\{c\in \CC\ |\
  {}^a({}^bc)=c\}|,\quad a,b\in\CC\] (the number of fixed points in
  $\CC$ under the iterated action shown). Proposition~\ref{quand}
  tells us (here $\sigma=\id)$ that if a rack has invertible Killing
  form then it is necessarily a quandle. The quadratic bialgebra
  $U(k\CC)$ is generated by $a\in \CC$ with relations $({}^ab)a=ab$
  for all basis elements $a,b\in \CC$.  If $\CC\subseteq
  G\setminus\{e\}$ is an ad-stable subset of a group $G$ it is
  well-known that it forms quandle (this point of view apparently goes back to
  Conway and Wraith), with ${}^ab=aba^{-1}$. In this case there is a
  bialgebra map $U(\CL)\to kG$ sending a basis element of $\CL$ to the
  same element viewed in $G$ (if $\CC$ generates then this is a
  surjection). Also in this case
  \begin{equation}\label{e:KViaChar}
  K(a,b) = |Z(ab)\cap\CC|=\chi_\CL(ab),\quad\forall a,b\in\CC
  \end{equation}
  where $\chi_\CL$ is the character of the conjugation representation
  of $G$ on $\CL$ and $Z(g)$ denotes the centralizer of $g\in
  G$. Clearly $K$ is ad-invariant since $\CC$ is and not only
  $\tilde\Psi$-symmetric but actually symmetric since
  $\chi_\CL(ba)=\chi_\CL(a(ba)a^{-1})=\chi_\CL(ab)$ for all
  $a,b\in\CC$. It is an interesting question if, starting
  with a `Lie algebra' $\CC\subseteq G\setminus\{e\}$ where $\CC$
  generates, we can recover the group $G$. The answer is in general that
  one has a covering group $G_\CC\twoheadrightarrow G$ \cite{Ma:cov}.
\end{example}

We have also made reference in the introduction to the use of `differential
calculus' on quantum groups as one method of construction of
braided-Lie algebras. We will not need this explicitly so suffice it
to say that a differential structure on a unital algebra $A$ means an
$A-A$-bimodule $\Omega^1$ of `differential 1-forms' equipped with a
map $\extd:A\to\Omega^1$ obeying the Leibniz rule. We also require
that the map $A\tens A\to\Omega^1$ sending $a\tens b\mapsto a\, \extd b$
is surjective and, optionally (one says that the calculus is
connected) that $\ker(\extd)$ is spanned by $1$. When $A$ is a Hopf
algebra or `quantum group' we can require the calculus to be covariant
under left or right translation, or both. In the latter case one says
that the calculus is bicovariant and these are classified by Ad-stable
right ideals $I$ in the augmentation ideal $A^+$
(cf. \cite{Woronowicz}). The left-invariant 1-forms $\Lambda^1$ can be
identified with $A^+/I$ and $\Omega^1$ is a free $A$-module over
$\Lambda^1$. The classical situation is where $A=\C[G]$, for
$G$ an algebraic group, $A^+=\mathfrak m_e$, the functions
vanishing at the identity, and $I=\mathfrak m_e^2$. 
Every unital algebra has a universal differential
calculus defined as $\Omega^1=\ker(\cdot:A\tens A\to A)$ and $\extd
a=1\tens a-a\tens 1$ which in the Hopf algebra case is bicovariant (it
corresponds to $I=0$). A calculus is called `inner' if there
is $\theta\in\Omega^1$ such that $[\theta,a]=\extd a$ for all $a\in
A$. This is a nonclassical concept.

\begin{theorem}\cite{GomMa} 
  Let $A$ be a coquasitriangular Hopf algebra and $\Omega^1$ an inner
  bicovariant differential calculus. Then there is a braided-Lie
  algebra $\CL$ associated to $\Lambda^1$ which lives in the braided
  category of right $A$-comodules.
\end{theorem}

In the case of the algebra of functions on a finite group the
bicovariant calculi are classified by ad-stable subsets $\CC\subseteq
G\setminus\{e\}$ with equality in the case of the universal calculus. The
calculus is inner with
\begin{equation}\label{theta}
	\theta=\sum_{a\in \CC}\omega_a
\end{equation}
where $\omega_a$ are the image in $A^+/I$ of the Kronecker
$\delta$-function at $a$. It is connected if $\CC$ generates. In
general, calculi on finite sets are classified by digraph structures
on the given set as the vertices, see \cite{Ma:gra} for some recent
work. The calculus is connected in the sense above if and only if the
underlying graph is connected. In this context $\CC$ stable under
inversion corresponds to the digraph being bidirected, i.e. with every
edge having arrows in both directions.  The graphs here are Cayley
graphs and are connected if and only if $\CC$ generates. From this
point of view:

\begin{lemma} 
  A finite group $G$ is simple {\em if and only if} all its nonzero 
  bicovariant calculi are connected.
\end{lemma}
\begin{proof} Suppose that $G$ is simple and $\CC$ a nonempty ad-stable subset
  (defining a nonzero bicovariant calculus). Let $N=\<\CC\>$ the subgroup
  generated by $\CC$. This is clearly normal and contains more than
  $e$ (as $\CC$ is nonempty), hence $N=G$ and the calculus is
  connected. Conversely, suppose that all nonempty ad-stable subsets
  $\CC$ generate $G$. Let $N\subseteq G$ be normal and
  $\CC=N\setminus\{e\}$. This is an ad-stable subset and $\<\CC\>=N$
  as $N\ne\{e\}$ is a normal subgroup, hence $N=G$. \end{proof}

Clearly one can further develop the differential geometry of `finite
Lie groups' and notably $S_3$ has in some sense constant curvature
while $A_4$ is Ricci flat for noncommutative differential structures
provided by suitable conjugacy classes and metrics
\cite{Ma:rieq,NML}. As far as we know no simple groups have yet been
studied at this level of noncommutative Riemannian geometry.

\section{Nondegeneracy of the Killing form for an ad-stable subset}

In this section we will look at the question of non-degeneracy of the
Killing form in maximum generality
and with miscellaneous results and examples. Then Section~4 will cover
the case of the universal calculus and Section~5 the case of conjugacy
classes. The analyses of these two extremal cases contain the main results of the paper.  

Let
$G$ be a finite group and $\CC\subseteq G\setminus\{e\}$ be an
ad-stable subset with $K$ the associated Killing form (\ref{killing})
on $\CL=k\CC$.

\begin{lemma}\label{lemZ} 
  If $\CC\cap (\CC.c)\ne\emptyset$ for some nontrivial $c\in Z(G)$
  then $K$ is degenerate. In particular, if $|Z(G)\cap\CC|>1$ then $K$
  is degenerate.
\end{lemma}
\begin{proof}
  Looking at $K$ as a matrix with rows and columns labelled by
  $\CC$. If $b=b'c$ where $b,b'\in \CC$ and $c\in Z(G)\setminus\{e\}$
  then $K(a,b)=|Z(ab)\cap\CC|=|Z(ab'c)\cap\CC|=|Z(ab')\cap\CC|=K(a,{b'})$ for all $a\in
  \CC$, hence $K$ has a repeated column.  If $b,b'\in Z(G)\cap\CC$ are
  distinct then $c=b'{}^{-1}b$ fits the first part.
\end{proof}

\begin{corollary}\label{ZG} 
  If $|G|>2$ and $|Z(G)|>1$ then the Killing form for the universal
  calculus is degenerate.
\end{corollary}
\begin{proof}
  For the universal calculus $\CC=G\setminus\{e\}$ and in the
  preceding lemma we can take any nontrivial $c\in Z(G)$, any $b'\ne
  c^{-1},e$ and $b=b'c$. Then $b\in\CC\cap(\CC c)$.
\end{proof}

If $G$ has order 2 then $K$ {\em is} a $1\times 1$ matrix and {\em is}
nondegenerate. We therefore only need to investigate the universal
calculus in the case where $|G|>2$ and $Z(G)=\{e\}$. Even for
nonuniversal calculi it will be necessary to avoid too much
intersection with $Z(G)$ as Lemma~\ref{lemZ} shows.

\begin{proposition}\label{D2n} 
  Let $D_{2n}$ be the dihedral group $\Z_n\rtimes \Z_2=\<a,x\>$ with
  relations $x^2=e$, $a^n=e$, $xa=a^{-1}x$.
  \begin{enumerate}
  \item For odd $n$, the universal calculus is nondegenerate
  \item For odd $n$, the order $2$ conjugacy classes
    $\CC_i=\{a^i,a^{-i}\}$, $i=1,\cdots,{n-1\over2}$ have degenerate
    $K$ and we have a single generating conjugacy class $\CC_0=\{  a^k x\, |\, 1\le k\le n\}$ of reflections, 
    which has nondegenerate $K$.
  \item For even $n$, the universal calculus has degenerate $K$.
  \end{enumerate}
  Hence for odd $n$, $D_{2n}$ is nondegenerate and class-nondegenerate
  in the sense of Definition~\ref{nondeg}.
\end{proposition}
\begin{proof} Part (3) is an application of Corollary~\ref{ZG} since
  when $n$ is even the centre is $\{e,a^{n\over 2}\}$. Now let us write 
  $\Z_n$ for the cyclic subgroup of rotations in $D_{2n}$. For part (1) 
  we use that the centralisers are
  \[ Z(a^i)=\begin{cases}\Z_n&{\rm if\ } i\ne 0\cr D_{2n}&{\rm
      otherwise}\end{cases},\qquad Z(a^ix)=\{e,a^ix\}.\] 
 Hence in the basis $\{a, a^2, \dotsc, a^{n-1}, x, ax,
  \dotsc, a^{n-1}x\}$ we have
 \begin{eqnarray*}
	K & = & \begin{pmatrix}
  (n-1)\theta_{n-1,n-1} + n \bar{1}_{n-1,n-1} & \theta_{n-1,n} \\
  \theta_{n,n-1} & (n-1)\theta_{n,n} + n1_{n,n},
\end{pmatrix}
\end{eqnarray*}
where $\theta_{i,j}$ is the matrix with $i$ rows and $j$ columns, and
all entries equal to  $1$, and $1_{j,j}$ is the $j\times j$ identity matrix, 
and $\bar 1_{j,j}$ is the permutation matrix which has 1's along the anti-diagonal  
and 0's 
elsewhere. 
The matrix $K$ is straightforwardly seen to be invertible with inverse
\[
K^{-1} = {1\over m n} \begin{pmatrix} -(n^2-n-1)\theta_{n-1,n-1}+ m
  \bar 1_{n-1,n-1} & -\theta_{n-1,n} \cr -\theta_{n,n-1} &
  -(n-1)^2\theta_{n,n}+m 1_{n,n}
\end{pmatrix},
\]
where $m=1-n^2+n^3$. 
 
 Note that (1) also follows later from
Theorem~\ref{roth} as one can easily see that $D_{2n}$ for $n$ odd has
the Roth property. Namely for $i\ne 0$ the conjugation representation $\C\CC_i$ 
decomposes as $\C\CC_i=1\oplus\bar 1$ where $\bar 1$ is
spanned by $a^i-a^{-i}$ so that conjugation by $x$ acts  as $-1$ and
$a$ acts trivially, and $1$ stands for the trivial representation. Meanwhile
$\C\CC_0=\C\theta\oplus(\bigoplus_{k=1}^{n-1\over 2}V_k)$ where
$\theta = \sum_{i=0}^{n-1} a^ix$ is the sum of all the elements in
$\CC_0$ and the $V_k$ are 2-dimensional irreps spanned by
$v^\pm_k=\sum_{j=0}^{n-1}e^{{\pi\imath jk\over n}}a^{\mp j}x$, where
conjugation by $x$ gives transposition and conjugation by $a$ gives
multiplication by a phase factor of $e^{2\pi \imath k\over n}$. 
Hence all $\frac {n+1}2$ irreps  occur
in the conjugation representation of the group  $D_{2n}$. 

For (2) on $\CC_i$ the entries of $K$ are 2, so these have degenerate 
$K$, while $\CC_0$ has
Killing form $K(a^ix,a^jx)=|\CC_0\cap Z(a^{i-j})|=n\delta_{ij}$ so
this is nondegenerate. The  $\CC_i$ classes do not generate,
while $\CC_0$ generates so the group is
class-nondegenerate. \end{proof}

Along the same lines it seems likely that $D_{2n}$ is in fact 
strongly nondegenerate for all
$n$ odd, and this has been experimentally 
confirmed for dihedral groups of small order, but a general proof 
would be significantly more
complicated than the universal case above and is deferred to
elsewhere. Clearly it is not absolutely nondegenerate.

 \begin{proposition}\label{sig} 
   If $\CGe$ has nondegenerate $K$ then the index of the latter is
   equal to the number of involutions in $\CC$.
 \end{proposition}
 \begin{proof} Let $\pi(a)$ be the matrix of $a\in\CC$ in the
   conjugation representation. As this is a permutation matrix, its
   inverse is its transpose. Hence if $a$ is an involution $\pi(a)$ is
   real and symmetric. We may identify $\C\CC$ with $\C^{|\CC|}$ using $\CC$
   as a basis. We denote by $\bar v$ the complex conjugate
 of $v\in \C\CC$ defined using this identification.  We also let $(\ )^\dagger$ denote the
   associated hermitian transpose.
 We decompose $\C\CC$ into `symmetric'
   and `anti-symmetric' parts, $\C\CC=S\oplus A$.  Here $S$ has
   a basis made up of the involutions in $\C$ and the elements $a+a^{-1}$ for all $a\in\CC$ not an
   involution, and  $A$ has a basis given by the non-zero elements of the form $a-a^{-1}$,
    which under $\pi$ go to real antisymmetric matrices. 
   Consider $v\in
   S$ and $w\in A$. Then $\pi(\bar v)=\pi(v)^\dagger$ because the
   basis elements go to real symmetric matrices, and therefore $K(\bar v,v)={\rm
     Tr}(\pi(v)^\dagger\pi(v))\ge 0$. Similarly $\pi(\bar
   w)=-\pi(w)^\dagger$ and $K(\bar w,w)=-{\rm
     Tr}(\pi(w)^\dagger\pi(w))\le 0$. Finally $K(\bar v,w)=0$ as the 
   trace of the product of  a symmetric and an antisymmetric matrix.
   If $K$ is nondegenerate then $K(\bar v,v)=0$ is not possible for
   $v\ne 0$ since the underlying real symmetric matrix of $K$ in our
   basis has no 0-eigenspace. Hence in this case $(\dim(S), \dim(A))$
   is the signature of $K$, their difference is the number of
   involutions.  \end{proof}

 Clearly, one has a similar result over $\R$ without complex conjugation.
 
  Next we associate some auxiliary objects to each ad-stable $\CGe$, namely
 \begin{equation}\label{thetaetc}
	\theta = \sum_{a\in\CC}a,\quad \theta^*(a) = 1,\quad\lambda^*(a) = |Z(a)\cap\CC|,\quad\forall a\in \CC
\end{equation}
where $\theta^*,\lambda^*$ are linear functions on $\C\CC$ defined on
basis elements. In a matrix-vector notation, we also define some
numerical constants in $\Q$ 
\begin{equation}\label{munu}
	\mu = \theta\cdot K^{-1}\theta,\quad \nu = \lambda \cdot K^{-1}\theta,\quad \rho = \lambda \cdot  K^{-1}\lambda,
\end{equation}
associated to any  $K$ for which $\theta$ and $\lambda$ are in the image.
Here we are working in our fixed basis $\CC$ and $\cdot$ is the dot
product of vectors in $\C\CC$ or Euclidean inner product in our
basis.  From elementary linear algebra\cite{Lei} the quantities 
$\mu,\nu$ and $\rho$ are well defined independent of the choice
of representatives $K^{-1} \theta$ and $K^{-1}\lambda$. 
For example, if $Kv=\theta$ and $K w=\lambda$
then 
\begin{eqnarray*} 
\nu:=\lambda^t v= w^t K^t v= w^t K v=w^t \theta,
\end{eqnarray*}
which is independent of the choice of $v$ because of its expression in terms
of $w$, and independent of $w$ because of its expression in terms of $v$.
In the case where $K$ is invertible, $\mu$ is the sum of all the entries of $K^{-1}$.


Also note that if $G_1,G_2$ are groups with $\CC_1,\CC_2$ ad-stable
subsets not containing the identity then
\[
\CC_1\times \CC_2,\quad \CC_1\sqcup\CC_2 =
\CC_1\times\{e\}\cup\{e\}\times\CC_2\subseteq G_1\times
G_2\setminus\{e\}
\]
satisfy the same properties in $G_1\times G_2$. The first of these has
\begin{equation}\label{Kprod}
	K_{\CC_1\times\CC_2}((a,b),(c,d)) = K_1(a,c)K_2(b,d),\quad \forall a,c\in \CC_1,\ b,d\in\CC_2
\end{equation}
and this is clearly nondegenerate if $K_i$ are.  The second (the
disjoint union) is the analogue of the direct sum of Lie algebra
structures on the direct product of Lie groups.

\begin{proposition}\label{Kdis} 
  Let $G_1,\CC_1$ and $G_2,\CC_2$ be two finite groups with ad-stable
  subsets $\CC_1,\CC_2$ and $K_1,K_2$ nondegenerate. Then
  $K_{\CC_1\sqcup\CC_2}$ is nondegenerate if and only if
  \[
  \det\begin{pmatrix} \rho_1 & d_2\nu_1 & 1 & \nu_1 \cr d_1\nu_2 &
    \rho_2 & \nu_2 & 1 \cr 1+d_1\mu_2 & \nu_2 & \mu_2 & 0 \cr \nu_1 &
    1+d_2\mu_2 & 0 & \mu_1
  \end{pmatrix}\ne 0
  \]
  where $d_i=|\CC_i|$ and $\mu_i,\nu_i,\rho_i$ are associated to $K_i$
  as in (\ref{munu}).
\end{proposition}
\begin{proof} When $a,b\in \CC_1$ we have
  $K(a,b)=|(\CC_1\sqcup\CC_2)\cap Z(ab)|=K_1(a,b)+d_2$ as all
  elements of the form $\{e\}\times\CC_2$ commute with
  $(ab,e)$. Similarly when $a\in \CC_1, b\in\CC_2$ we will have
  $K(a,b)=\lambda_1(a)+\lambda_2(b)$. In block matrix form this looks
  like
  \[
  K_{\CC_1\sqcup\CC_2} =
  \begin{pmatrix}
    K_1 & 0 \cr 0 & K_2
  \end{pmatrix} +
  \begin{pmatrix}
    \theta_1\tens\theta_1^* d_2\ \ & \lambda_1\tens
    \theta_2^*+\theta_1\tens\lambda_2^* \cr
    \lambda_2\tens\theta_1^*+\theta_2\tens\lambda_1^* \ \ & d_1
    \theta_2\tens\theta_2^*
  \end{pmatrix}\] Hence for $v+w\in \C\CC_1\oplus\C\CC_2$ as a column
  vector to be in the kernel means
  \[
  K_1v+\theta_1 (d_2 \theta_1\cdot v + \lambda_2\cdot w) +
  \lambda_1\theta_1\cdot w = 0, \quad K_2w + \theta_2(d_2
  \theta_2\cdot w + \lambda_1\cdot v) = 0.
  \]
  When the $K_i$ are invertible we write these as
  \begin{equation}\label{vw}
	v + K_1^{-1}\theta_1(d_2\alpha+\delta) + K_1^{-1}\lambda_1\beta = 0, 
	\quad w + K_2^{-1}\theta_2(d_1\beta + \gamma) + K_2^{-1}\lambda_2\alpha = 0
\end{equation}
where
\[
\alpha = \theta_1\cdot v,\quad \beta = \theta_2\cdot w,\quad \gamma =
\lambda_1\cdot v,\quad \delta=\lambda_2\cdot w.
\]
We now apply $\theta_i\cdot$ and $\lambda_i\cdot$ to (\ref{vw}) to obtain
four equations for these four scalars, described by the displayed
matrix.  Hence $\alpha,\beta,\gamma,\delta$ are zero and hence $v,w$
are zero, unless the stated determinant is not zero.  Conversely if
the determinant vanishes we may solve for $\alpha,\beta,\gamma,\delta$
hence for $v,w$ and $K$ is degenerate. \end{proof}

The result suggests that `generically', i.e. unless the determinant
accidentally vanishes, nondegenerate Killing forms remain
nondegenerate for the direct sum `Lie algebra' structure on the direct
product of two groups.
\begin{example}\label{Dnmu}
  For $D_{2n}$ with $n\ge 3$ odd and the universal calculus, we have $
  \lambda=(n-1,\cdots,n-1,1,\cdots, 1)$ and
  \[
  \begin{array}{rclrcl}
  d & = & 2n-1, \qquad & \mu & =& \displaystyle \frac{3-4 n+2 n^2}{1-n^2+n^3},  \vspace{1ex} \\
  \nu & =& \displaystyle \frac{1+n-2 n^2+n^3}{1-n^2+n^3}, \qquad&\rho & =&  \displaystyle \frac{-1+2 n+2 n^2-3 n^3+n^4}{1-n^2+n^3}
\end{array}
\]
from formulae in the proof of Proposition~\ref{D2n}. Then any two such
$D_{2n}, D_{2n'}$ have nondegenerate $D_{2n}\times D_{2n'}$ with the
direct sum `Lie algebra'. Here the determinant can be checked for
small $n,n'$ while at least for large (in fact all) $n,n'\ge 3$ the
determinant grows more negative as either $n,n'$ increase and hence is
never zero. \end{example}

We will study the determinant criterion further in Section~5 in the
case where $\CC_i$ are conjugacy classes.

\section{Nondegeneracy for the universal calculus}

In this section we will exclusively study the case
$\CC=G\setminus\{e\}$ where $G$ is a finite group, i.e. the universal
calculus and its associated `Lie algebra' structure on $G$. We have
already seen in Corollary~\ref{ZG} that for a group to be
nondegenerate in the sense of Definition~\ref{nondeg} it will at least
have to be centreless or $\Z_2$.

In the present setting of the universal calculus we have
\[
K(a,b)=|Z(ab)|-1,\quad \forall a,b\in \CC
\]
and we similarly have
\[ \lambda^*(a)=|Z(a)|-1,\quad \forall a\in \CC\] which we extend as a linear function on $k\CC$. This is the character of the
conjugation representation on $\CC$ restricted to $\CC$. Also in our case the
`inner generator' (\ref{thetaetc}) is
\[ \theta=\Lambda - e\] where $\Lambda=\sum_{g\in G}g$ is the integral
in the group algebra. Similarly $ \theta^*(a)=1$ for all $a\in \CC$ is
the integral on the group regarded as a linear function on $\C \CC$.

\begin{lemma}\label{munuN} For the universal calculus on a finite
  group,
  \[ K(\theta,\ )=-\lambda^* + |G|(N-1)\theta^*\] where $N$ is the
  number of conjugacy classes in $G$ including the trivial one. If $\lambda$ or $\theta$ are in the image
  of the  Killing form  then the associated constants $\mu,\nu,\rho$ in (\ref{munu}) are defined and related by
  \[ 1-\nu=|G|(1-(N-1)\mu),\quad |G|-1-\rho=|G|^2(N-1)(1-(N-1)\mu).\]
\end{lemma}
\begin{proof} We consider $\lambda^*$ to be defined on all of $G$ by the
  formula (the trace of the conjugation representation on
  $\C\CC$). Then for $h\in G$,
  \[ K(\theta,h)=\sum_{g\ne e}\lambda^*(gh)=\sum_{g\ne
    h}\lambda^*(g)=-\lambda^*(h)+\sum_{g\in G}\lambda^*(g)\] so that
  $K(\theta,\ )=-\lambda^*+\lambda^*(\Lambda)\theta^*$. Moreover,
  $\lambda^*(\Lambda)=(\sum_{g\in G}|Z(g)|)-|G|=|G|(N-1)$ by the orbit
  counting lemma. Also note that $\lambda^*(\theta)=|G|(N-2)+1$. This gives the displayed formula. Clearly then if either $\theta$ or $\lambda$ are in the image of 
  $K$ then so is the other so that $\mu,\nu,\rho$ are well-defined. Using a vector-matrix notation, we have
  $\theta=-K^{-1}\lambda+|G|(N-1)K^{-1}\theta+m$ for any choice of inverse elements and some $m\in \ker K$ and applying 
  $\theta\cdot$ and $\lambda\cdot$ to this to give relations
  \[d+ \nu=|G|\mu(N-1),\quad |G|(N-1)+\rho-d=|G|(N-1)\nu\] which we
  write as stated. Here $\theta.m=\lambda.m=0$ by writing $\theta,\lambda$ as in the image of $K$ and using the symmetry of $K$ to move it over to operate on $m$. 
  (Note also that $(\nu-1)(\nu+d)=(\rho-d)\mu$ as a
  consequence). Of course $d:=|\CC|=|G|-1$ in the present context.
\end{proof}

Our main result of this section is the following theorem.

\begin{theorem} \label{roth} Let $G$ be a finite group such that the
  conjugation representation on $\C G$ contains every irrep (`Roth
  property'). Then it is nondegenerate.
\end{theorem}
We will prove this by proving a more general result,
Proposition~\ref{rothW}. Here we work with the expression
\eqref{e:KViaChar}. We note that this formulation of the 
Killing form via the character of the conjugation representation $\C\CC$ 
makes sense for any representation $W$ as a bilinear form on $\C G$. Namely we let
\[
K_W(a,b) := \chi_W(ab)
\]
where $\{a\ |\ a\in G\}$ is a basis of $\C G$. It is well-known that
this symmetric bilinear form on $\C G$ is nondegenerate if and only if $W$
contains every irrep of $G$ with positive multiplicity. This follows
from semisimplicity of the group algebra $\C G$ and general facts
about semisimple algebras. Namely, if an algebra is semisimple then it
is a direct sum of matrix blocks. If $W=\bigoplus_i n_iV_i$ for some
multiplicities $n_i$ of irreps $V_i$ then $K_W$ has a block form with
$n_i$ times the Euclidean inner product on each matrix block. This is
because an element in a matrix block corresponding to a particular 
irrep acts as zero by left multiplication on any other block and hence 
in any other irrep. Hence $K_W$ is nondegenerate on $\C G$ if and only if all the $n_i>0$.

\begin{proposition}\label{rothW}
  Let $G$ be a finite group and $W$ be a representation of $G$ for
  which every irreducible representation appears in $W$ with strictly
  positive multiplicity. Then the restriction of $K_W$ to
  $\C.(G\setminus \{e\})$ is nondegenerate.
\end{proposition}
\begin{proof}
  Since the form $K_W$ is nondegenerate on $\C G$ we know that
  $(\C.(G\setminus\{e\}))^{\perp}$ is one-dimensional. We will know
  that $K_W$ is nondegenerate in $\C.(G\setminus\{e\})$ if there is no
  element in the perpendicular which also lies
  $\C.(G\setminus\{e\})$. We prove this by determining explicitly a
  vector spanning the line $(\C.(G\setminus\{e\}))^{\perp}$, and
  observing that it doesn't lie in $\C.(G\setminus\{e\})$.

  Suppose the irreps are $V_1,\cdots,V_n$ say. Define
$$
m=\sum_{g\in
  G}\left(\sum_{i=1}^{n}\frac{\dim(V_i)^2}{\left<\chi_{V_i},\chi_{W}\right>}\overline{{\chi}_{V_i}(g)}\right)
g
$$
This is well-defined since $\left<\chi_{V_i},\chi_{W}\right>\ne 0$ for
all $i$ by our assumption. We claim that $K_W(a,m)=0$ for all $a\ne
e$. Note that the coefficient of $e$ in $m$ is given by the formula
$$
m_e= \sum_{i=1}^{n}\frac{\dim(V_i)^3}
{\left<\chi_{V_i},\chi_{W}\right>},
$$
so is always strictly positive. Therefore $m$ does not lie in
$\C.(G\setminus\{e\})$ and the claim will imply the proposition. We
recall the standard orthogonality relations
\begin{equation*}\label{e:OR1}
\sum_{g\in G}\overline{\chi_V(g)}\ \chi_{V'}(ga)=\begin{cases}0&\text{if $V,V'$ are distinct irreps,}\\
  \frac{|G|}{\dim V}\chi_V(a) &\text{otherwise},
\end{cases}
\end{equation*}
and
\begin{equation*}\label{e:OR2}
\sum_{i}\overline{\chi_{V_i}(a_1)}\chi_{V_i}(a_2)=0\quad\text{if $a_1$ and $a_2$ are not conjugate}
\end{equation*}
and also note that
$$
\chi_W(g)=\sum_{i}\left<\chi_{V_i},\chi_W\right>\chi_{V_i}(g).
$$
Now we can compute, extending linearly,
\begin{multline*}
  K_W(m, a)=K_W\left(\sum_{g\in G}\left(\sum_{i=1}^{n}\frac{\dim(V_i)^2}{\left<\chi_{V_i},\chi_{W}\right>}\overline{{\chi}_{V_i}(g)}\right) g, a\right)\\
  =\sum_{g\in
    G}\left(\sum_{i=1}^{n}\frac{\dim(V_i)^2}{\left<\chi_{V_i},\chi_{W}\right>}\overline{{\chi}_{V_i}(g)}\right)
  \chi_W(ga)=
  \sum_{i=1}^{n}\frac{\dim(V_i)^2}{\left<\chi_{V_i},\chi_{W}\right>}\left(\sum_{g\in G}\overline{{\chi}_{V_i}(g)} \chi_W(ga)\right)\\
  =\sum_{i=1}^{n}\dim(V_i)^2\left(\sum_{g\in
      G}\overline{{\chi}_{V_i}(g)} \chi_{V_i}(ga)\right)=
  | G| \sum_{i}\dim(V_i)\chi_{V_i}(a)\\
  =|G|\sum_{i}\overline{\chi_{V_i}(e)}\chi_{V_i}(a).
\end{multline*}
The last expression vanishes whenever $a\ne e$ by orthogonality.
\end{proof}

\begin{proof} (of the theorem) Now suppose that $G$ is a finite group
  and  the conjugation representation on $\C G$ with contains every
  irrep (this is the Roth property in this context). We set
  $W=\C.(G\setminus\{e\})$ where we remove the group identity. This $W$
  still contains a copy of the trivial representation since it is a 
  permutation representation, and for example the element $\theta$ is
  invariant. As $\C G= W\oplus\C.e$ as a $G$-module, any nontrivial
  representation contained in $\C G$ must also be in
  $W$. Hence $W$ also enjoys the property of containing every
  irrep. We then apply Proposition~\ref{rothW}. \end{proof}

Note that the group will have to be centreless for the Roth property
to hold in the form stated. The simplest example where the Roth
property holds is $G=S_3$, the group of permutations on three
elements.  This is elementary enough that we can, instructively,
work out everything in our above approach by hand.

\begin{example}\label{KexampleS3} Let $G$ be $S_3$, with its three irreducible characters
$\chi_{triv},\chi_{sign}$ and $\chi_{\Delta}$ corresponding to the 
trivial, the sign and the standard representations, $V_{triv},V_{sign}$,$V_{\Delta}$.
Then it is easy to see that the conjugation representation 
$W=\C (G\setminus\{e\})$ decomposes as $V_{triv}^{\oplus 2}\oplus V_{sign}\oplus V_{\Delta}$,
so it contains ever irrep of $G$.  
The Killing form on $\C G$ for $W=\C (G\setminus\{e\})$ is
  \[ K_W=\begin{pmatrix} 5& 1& 1& 1& 2& 2\\ 1& 5 & 2 & 2& 1 & 1\\ 1 &
    2 & 5 & 2& 1 & 1\\ 1 & 2 & 2 & 5 & 1 & 1\\ 2 & 1& 1& 1& 2& 5\\ 2&
    1& 1& 1& 5& 2\end{pmatrix}\] in a basis $e, u=(12), v=(23),
  w=(13)=uvu, uv=(123)$ and $vu=(132)$. This matrix is just obtained 
  by working out $\chi_W(g h)=|Z(g h)|-1$ for all $g,h\in S_3$.  
  One can then see by direct computation that the
  lower right $5\times 5$ block in $K_W$ is invertible as required by
  Theorem~\ref{roth}. 
  Or, like in proof of Proposition~\ref{rothW}, $\C(G\setminus\{e\})^\perp$  is spanned by the single group algebra element,
  $$
  m=\frac 1 2\chi_{triv}+\chi_{sgn}+4\chi_\Delta=\frac{1}{2}(19 e-(u+v+w)-5 (uv+vu)).
  $$
  Clearly, since
  $m$ has a nonzero coefficient of $e$, it
   does not lie in $\C.(G\setminus\{e\})$.

\end{example}

Next, the Roth property is manifestly closed under direct
products. We see now how this emerges from linear algebra in our
Killing form approach and that the converse holds.
 
\begin{proposition}\label{produni} Let $G_1,G_2$ be two nondegenerate
  finite groups. Then $G_1\times G_2$ is nondegenerate {\em if and
    only if}
  \[ \mu\ne {1\over N-1}\] for each group.
  \end{proposition}\begin{proof} The Killing form for the universal
  calculus on $G_1\times G_2$ is
  \begin{eqnarray*}K((a,b),(c,d))&=&| Z_{G_1\times G_2}((ac,bd))|-1=|Z_{G_1}(ac)||Z_{G_2}(bd)|-1\\
&=&\tilde K_1(a,c)\tilde K_2(b,d)+\tilde K_1(a,c)+\tilde
 K_2(b,d)\end{eqnarray*} in terms of the extensions of Killing forms
 $K_i$ of each group to $\C G_i$. We write
 \[
 \CC = (G_1\times G_2)\setminus\{(e,e)\} = (\CC_1\times\CC_2) \sqcup
 \CC_1\sqcup\CC_2
  \]
 where $\CC_i=G_i\setminus\{e\}$. We have $K$ then in $3\times 3$
 block form with

\begin{eqnarray*}
K((a,b),(c,d)) & = & K_1(a,c)K_2(b,d)+K_1(a,c)+K_2(b,d) \\
K((a,b),(c,e)) & = & K_1(a,c)(1+\lambda_2(b))+\lambda_2(b) \\
K((a,b),(e,d)) & = & K_2(b,d)(1+\lambda_1(a))+\lambda_1(a) \\
K((a,e),(c,e)) & = & K_1(a,c)(d_2+1)+d_2 \\
K((a,e),(e,d)) & = & \lambda_1(a)\lambda_2(d)+\lambda_1(a)+\lambda_2(d)
\end{eqnarray*}
for $a,c\in\CC_1, b,d\in\CC_2$, and the other cases by symmetry. In
matrix-vector notation, for an element $v+w+z\in
\C(\CC_1\times\CC_2)\oplus\C\CC_1\oplus\C\C_2$, written as a matrix, a
column vector and a row vector respectively, to be in the kernel of
$K$ means
\[ K_1vK_2+K_1v\theta_2\theta_2^*+\theta_1\theta_1^*vK_2+K_1w
(\lambda_2^*+\theta_2^*)+\theta_1\theta_1^*w\lambda_2^*+(\lambda_1+\theta_1)z
K_2+\lambda_1z\theta_2\theta_2^*=0\]
\[
K_1v(\theta_2+\lambda_2)+\theta_1\theta_1^*v\lambda_2+|G_2|K_1w+d_2\theta_1\theta_1^*w+(\lambda_1+\theta_1)z\lambda_2+\lambda_1
z\theta_2=0\]
\[
(\theta_1^*+\lambda_1^*)vK_2+\lambda_1^*v\theta_2\theta_2^*+\lambda_1^*w(\lambda_2^*+\theta_2^*)+\theta_1^*w\lambda_2^*+|G_1|zK_2+d_1
z\theta_2\theta_2^*=0.\] We will use the notation
\[ \phi=\theta_1\cdot w,\quad \psi=z\theta_2,\quad
\sigma=\lambda_1\cdot w,\quad \tau=z\lambda_2\]
\[\alpha=\theta_1\cdot v\lambda_2,\quad \beta=\lambda_1\cdot
v\theta_2,\quad \gamma=\theta_1\cdot
v\theta_2,\quad\delta=\lambda_1v\cdot\lambda_2\] then applying
$K_i^{-1}$ our three equations for the kernel become
\[
v + v\theta_2\theta_2^*K_2^{-1} + K_1^{-1}\theta_1\theta_1^*v +
w(\lambda_2^* + \theta_2^*) K_2^{-1} +
K_1^{-1}\theta_1\phi\lambda_2^*K_2^{-1} +
K_1^{-1}(\lambda_1+\theta_1)z +
K_1^{-1}\lambda_1\psi\theta_2^*K_2^{-1} = 0
\]
\[
v(\theta_2+\lambda_2) + K_1^{-1}\theta_1\alpha +
|G_2|w+d_2K_1^{-1}\theta_1\phi + K_1^{-1}(\lambda_1 + \theta_1)\tau +
K_1^{-1}\lambda_1 \psi = 0
\]
\[
(\theta_1^* + \lambda_1^*)v + \beta\theta_2^*K_2^{-1} +
\sigma(\lambda_2^*+\theta_2^*)K_2^{-1} + \phi\lambda_2^*K_2^{-1} +
|G_1|z + d_1 \psi\theta_2^*K_2^{-1} = 0.
\]
We now apply evaluation or dot product of the relevant
$\theta_i,\lambda_i$ to the two sides of the first equation and to one
side of each of the other two, to obtain eight equations for the
scalar variables $\phi,\psi,\beta,\alpha,\sigma,\tau,\gamma,\delta$
governed in that basis order by an $8\times 8$ matrix. Specifically,
if
\[
\begin{vmatrix}
  {d_1} \nu_1 & \rho_1 & 1 & \nu_1 & 1+{d_2} & \nu_1+\rho_1 & 0 & 1 \\
  \rho_2 & {d_2} \nu_2 &\nu_2 & 1 & \nu_2+\rho_2 & 1+{d_1} & 0 & 1 \\
  \nu_1 \nu_2 & \nu_1+(1+\mu_2) \rho_1 & 1+\mu_2 & 0 & \mu_2+\nu_2 & 0 & \nu_1 & 0 \\
  \nu_2+(1+\mu_1) \rho_2 & \nu_1 \nu_2 & 0 & 1+\mu_1 & 0 & \mu_1+\nu_1 & \nu_2 & 0 \\
  \nu_2 & 1+{d_1}(1+ \mu_2) & 1+\mu_2 & 0 & \mu_2+\nu_2 & 0 & 1 & 0 \\
  1+{d_2}(1+ \mu_1) & \nu_1 & 0 & 1+\mu_1 & 0 & \mu_1+\nu_1 & 1 & 0 \\
  \mu_2+(1+\mu_1) \nu_2 & \mu_1+(1+\mu_2) \nu_1 & 0 & 0 & 0 & 0 & 1+\mu_1+\mu_2 & 0 \\
  \nu_1 \rho_2 & \nu_2 \rho_1 & \nu_2 & \nu_1 & \nu_2+\rho_2 &
  \nu_1+\rho_1 & 0 & 1
\end{vmatrix}
\ne 0
\]
then these variables are all zero and our three equations for $v,w,z$
simplify to
\begin{equation}\label{veqn}
  v+v\theta_2\theta_2^*K_2^{-1}+K_1^{-1}\theta_1\theta_1^*v+w
  (\lambda_2^*+\theta_2^*)K_2^{-1}+K_1^{-1}(\lambda_1+\theta_1)z = 0
\end{equation}
\[
v(\theta_2+\lambda_2)+|G_2|w = 0,\quad
(\theta_1^*+\lambda_1^*)v+|G_1|z = 0.
\]
The determinant here factorises with degree 2 factors
\begin{equation}\label{roth0}
(\rho-d)(1+\mu)- (\nu-1)^2
\end{equation}
for each group, so we require these not to vanish, which given
Lemma~\ref{munuN} we write as $\mu\ne 1/(N-1)$ on each group. In this
case one can eventually solve the linear system to determine that $v=w=z=0$ so that $G_1\times G_2$ is
nondegenerate. Details are omitted in view of our later Theorem~\ref{rothmu} which shows 
that both groups are Roth and hence so is their direct product, after which we can
use Theorem~\ref{roth}. Conversely, if $G_1$ alone obeys $\mu_1(N_1-1)=1$
so $\nu_1=1$ and $\rho_1=d_1$, then the displayed matrix has a
1-dimensional null space spanned by $\alpha=-\tau, \delta=-d_1\tau$
and $\beta=\gamma=\phi=\psi=\sigma=0$. We then solve for the vector
variables to find (for example) $x=-z, y=w=0$, $s=-d_1z$,
$t=-K_1^{-1}\lambda_1 z$ provided $z\theta_2=0$ and
$z\lambda_2=\tau$. Then $v=-K_1^{-1}\lambda_1 z$ from the equation for
$v$. We then check that any $z$ such that $z\theta_2=0$ and $w,v$ as
stated reproduce all other vectors and scalars as stated and thereby
that the equations to be in the kernel are satisfied. Hence $K$ is
degenerate. If both $G_1,G_2$ obey $\mu(N-1)=1$ then the kernel of the
displayed matrix is 2-dimensional but includes the previous one. The
solution above still applies and $K$ is degenerate.
\end{proof}

As noted in the proof, we will see in the next theorem that $\mu\ne \frac{1}{N-1}$ is characteristic of a Roth property 
group. So this proposition says that 
within the class of non-degenerate groups the Roth property
groups form the largest subclass which is closed under 
direct products. For example for $\Z_2$ is nondegenerate and not Roth thence the direct product $\Z_2\times G$ with any
nondegenerate group $G$ will be degenerate. This latter result also follows from
Corollary \ref{ZG} as $\Z_2 \times G$ has a nontrivial center.

One can also compare with Proposition~\ref{Kdis} for the disjoint union of
universal `Lie algebra' structures on the direct product of two
nondegenerate finite groups. Here we find by contrast that if $G_1$ is
non-Roth then $G_1\times G_2$ with the disjoint union structure has
nondegenerate Killing form if and only if $G_2$ is Roth.  Thus for
example $\Z_2\times G$ for $G$ any Roth property group and with the
disjoint union {\it will} have nondegenerate Killing form. 

We now give our second  main
result of the section, which is the mentioned complete characterisation of when a finite group is Roth in terms of the Killing form, irrespective
of nondegeneracy.

\begin{theorem}\label{rothmu} Let $G$ be a nontrivial finite 
  group and $N$ the number of conjugacy classes.  The constants $\mu,\nu,\rho$ associated
 in \eqref{munu} to the Killing form for the universal calculus are well-defined. Moreover, the following are equivalent.
  \begin{enumerate}
  \item $G$ has the Roth property
  \item $\mu\ne \frac{1}{N-1}$
  \item $\nu\ne 1$
  \item $\rho\ne |G|-1$.
  \end{enumerate}
  In the Roth case
  \[
  \mu = {1\over n_0}\left(1-{1\over n_0\sum_j{ \dim(V_j)^3\over
        n_j}}\right)
  \]
  where $n_j$ is the multiplicity of irrep $V_j$ in the representation
  on $W=\C.(G\setminus\{e\})$ and $n_0=N-1$ is the multiplicity of
  the trivial representation.
 \end{theorem}
\begin{proof} (i) Using the same methods as in the proof of Lemma~\ref{munuN}, we 
regard the characters $\chi_{V_i}$ by restriction as vectors with entries the
  $|G|-1$ values at the different points of $\CC$. Then
  \begin{equation}\label{Kchi}
K(\overline{\chi_{V_j}},h) = 
\sum_{g\in\CC}\chi_{V_j}(g^{-1})\lambda^*(gh) = 
-\lambda^*(h)\dim(V_j)+{n_j |G|\over\dim(V_j)}\chi_{V_j}(h)
\end{equation}
using the orthogonality of characters. We will need this formula.

 (ii) If $G$ is missing no irreducible representations in its conjugation representation then we know 
that $K$ is invertible by Theorem~\ref{roth} so that $\mu,\nu,\rho$ are defined.
Alternatively, suppose $G$ is missing $V_1$, say. The formula \eqref{Kchi} tells us that 
  $$
  K(\overline{\chi_{V_1}})=-\dim(V_1)\lambda.
  $$ 
  so  $\lambda$ is in the image of $K$ in this case. 
Lemma~\ref{munuN} then implies that $\theta$ is also in the image of $K$, namely
$$
K( \theta - \frac{\overline{\chi_{V_1}}}{\dim(V_1)})=|G| (N-1) \theta.
$$ 
Hence Lemma~\ref{munuN} applies, $\mu,\nu,\rho$ are well defined for any finite group
and are related by the formulae stated there. This also means that (2)-(4) are all equivalent.

(iii) In the non-Roth case, the vector $-\overline{\chi_{V_1}}/\dim(V_1)$ lies in the inverse image $K^{-1}\lambda$. Here $V_1$ is, as in (ii).  We take the dot-product with  $\lambda$ and use orthogonality of characters to find 
\begin{multline*} \rho=-\frac{1}{\dim(V_1)}\sum_{g\ne e}\lambda(g)\overline{\chi_{V_1}}(g)=
\\
= -\frac{1}{\dim(V_1)}\left(\sum_{g\in G} \chi_{\C G\setminus\{e\}}(g)\overline{\chi_{V_1}}(g)\right)+ |G\setminus\{e\}| = |G|-1,
\end{multline*}
 because $V_1$ does not occur in the conjugation representation on $\C (G\setminus\{e\})$, and
 where we rewrote  $\lambda(g)=\chi_{\C G\setminus\{e\}}(g)$.

 
(iv) In the Roth case  since
$K$ is invertible, the formula \eqref{Kchi} in vector notation gives
\begin{equation}\label{Kinvchi}
\overline{\chi_{V_j}} = - \dim(V_j)K^{-1}\lambda+ {n_j |G|\over\dim(V_j)}K^{-1}\chi_{V_j}.
\end{equation}
We multiply both sides by $\dim(V_j)^2/n_j$ and sum over $j$. Now
the right hand summand becomes $|G|\, K^{-1}(\sum_j\dim(V_j)\chi_{V_j})$. But
 $\sum_j\dim(V_j)\chi_{V_j}$ is the character of the regular 
 representation and has support only on $e$. Hence 
  regarded by restriction as a vector in $\C\CC=\C (G\setminus\{e\})$, this is zero. Now
taking the dot product with $\lambda$ we have
\[
-\left(\sum_j{\dim(V_j)^3\over n_j}\right)\rho = \sum_j
{\dim(V_j)^2\over n_j}\lambda\cdot\overline{\chi_{V_j}} = \sum_j
{\dim(V_j)^2\over n_j}(n_j |G|-|\CC|\dim(V_j))
\]
which gives $|\CC|-\rho=|G|^2/(\sum_j\dim(V_j)^3/n_j)$ and hence the
formula for $\mu$ using Lemma~\ref{munuN}. \end{proof}
 
  \begin{example} For the Roth property group $S_3$ the formula
 for $\mu$ above gives that $\mu=\frac{9}{19}$, which agrees with what 
 we already computed in Example~\ref{Dnmu} (since $S_3\cong D_{6}$)
 and can also easily be checked from Example~\ref{KexampleS3}.  
 \end{example}

Theorem~\ref{rothmu} characterises the Roth property groups among
all finite groups in our Killing form approach. All nontrivial finite abelian groups are non-Roth and
it is easy to see that 
$\mu=\frac{1}{N-1}$. The first non-Roth centreless nondegenerate group is the small
group with label $(400, 207)$ (cf \cite{SG}) of order 400 as mentioned
in the introduction. The first finite simple nonabelian non-Roth group is $PSU(3,3)$,
which is not nondegenerate but where $\mu=\frac{1}{N-1}$ still applies as can be checked. 
Also, an immediate consequence for the
formula for $\mu$ is that for Roth property groups
\[ {1\over N} < \mu < {1\over N-1}\]
as follows from the observation for all $j\ne 0$ that $0< \dim(V_j)$,
$n_j<|G|-1=\sum_{j\ne 0}\dim(V_j)^2$. One can do better here, for
example these observations actually imply $\mu<1/(N-1+ {{1\over( |G|-1)^2}})$. In
the case of $D_{2n}$ with $n$ odd using the results in
Example~\ref{Dnmu} one finds $\mu\to {1\over{(N-1)}}$ strictly from below as
$n\to \infty$.

Going the other way, when the group is not Roth we can still say something
about the Killing form.

\begin{proposition}\label{deg} 
  Let $G$ be a finite group. $K$ for the universal calculus is nondegenerate on the
  subspace of invariant vectors inside $\C.(G\setminus\{e\})$ iff the conjugation representation on $\C G$ is missing at most one irreducible representation. \end{proposition}
\begin{proof} (i) Suppose $G$ is missing two distinct irreps, say $V_1,V_2$. In part (ii) of the proof of Theorem~\ref{rothmu} we 
have $-\overline{\chi_{V_1}}/\dim(V_1)$ and now also  $-\overline{\chi_{V_2}}/\dim(V_2)$ are in the preimage of $\lambda$. As the irreps are non-equivalent their characters are
linearly independent and hence $K$ has a kernel, even when restricted to the subspace of invariant vectors. 

(ii) We return to the formula (\ref{Kchi}) in vector form and
  suppose that $v=\sum_{i=1}^n v_i\overline{ \chi_{V_i}}$ where we
  omit $\overline{\chi_{V_0}}$ and keep the rest as basis of the
  `class vectors' $Z=(\C \CC)^{Ad}$ of vectors invariant under
  conjugation. We let $\delta=(\dim(V_i))$ be the vector of dimensions
  in this basis so $\delta_i=\dim(V_i)$ for $i=1\cdots N-1$. Then
  $Kv=0$ is equivalent to
  \[
  (v\cdot \delta)\lambda = |G|\sum_iv_i {n_i\over\delta_i}\chi_{V_i}
  \]
  but $\lambda=\sum_{i=0}^{N-1}
  n_i\chi_{V_i}=\sum_{i=1}^{N-1}(n_i-n_0\delta_i)\chi_{V_i}$ so this
  is equivalent to
  \[
  n_i (|G|
  v_i-\delta_i(v\cdot\delta))+n_0\delta_i^2(v\cdot\delta)=0,\quad\forall
  i=1\cdots N-1.
  \]
  Now if $n_1=0$ then $v\cdot\delta=0$. Putting this information into the displayed
  equation with the assumption $n_i\ne 0$ for $i>1$ gives $v_i=0$ for $i>1$. In this case  $v\cdot\delta=0$ tells us that $v_1=0$ as well, i.e. $v=0$. 
  Hence $K|_Z$ is  nondegenerate, but could still be degenerate on all of $\C
  \CC$. \end{proof}

In summary, if no irreps are missing in the conjugation representation then the group is
nondegenerate. If two or more irreps are missing then the group is not
nondegenerate as the Killing form for the universal calculus is degenerate. As far as we know the case of one irrep missing can go
either way but if the group is nondegenerate but not Roth then it must
have precisely one irrep missing. Using the methods of \cite{Pass} one can see that the group $PSU(3,4)$
is indeed missing exactly one irrep and this is now proven \cite{Tiep} to hold for all
finite simple nonabelian non-Roth groups. So the above proposition applies and we have the immediate
corollary:

\begin{corollary} Let $G$ be a finite simple nonabelian group. Then the Killing form for the universal calculus on $G$ is nondegenerate on the subspace of conjugation-invariant vectors inside $\C.(G\setminus\{e\})$. \end{corollary}

\section{Nondegeneracy, eigenvalues and reducibility of the Killing
  form for conjugacy classes}

In this section we will be interested in $\CC$ a conjugacy class but we
start off more generally. Let $G$ be a finite group and $\CGe$ an
ad-stable subset. Since we have a particular basis for $\CL=\C\CC$ we
have already had occasion to regard this for convenience as an
operator
\[
K:\CL\to \CL,\quad K(a) = \sum_{b\in \CC}K(a,b)b,
\]
and we now look at its properties as such in more detail. Note that by
construction this operator is ad-invariant and hence its eigenspaces
provide a natural decomposition of $\C \CC$ into subrepresentations.
Nondegeneracy in this language means of course that $K$ has no zero
eigenvalues in its spectrum. As $K$ is real and symmetric in our basis
it can be diagonalised over $\R$.
However, it can also be viewed as a hermitian matrix or self-adjoint
operator over $\C$. Moreover, the entries of $K$ are non-negative
integers. We give some basic consequences of these properties here.

\begin{proposition} Suppose $V$ is an irreducible representation of
  $\C G$ which is defined over $\Q$. So $V=V_\Q\otimes_\Q \C$ for an
  irreducible representation $V_\Q$ of $\Q G$. Furthermore, suppose
  $\mathcal C$ is a conjugacy class in $G$ such that $V$ occurs in the
  conjugation representation $\C\mathcal C$.

  If the isotypical component of $V$ in $\C\mathcal C$ is contained in
  a single eigenspace of the Killing matrix $K$ of $\mathcal C$, then
  the corresponding eigenvalue lies in $\Z$.
\end{proposition}

\begin{proof} 
  Choose an element $x\in \mathcal C$ and consider the map
  \[
  \pi: \C G\to \C\mathcal C: g\mapsto gxg^{-1}
  \]
  which is a $G$-equivariant surjection from the left-regular
  representation to the conjugation representation of $G$.  Let
  $A_V\subset \C G$ denote the block of the irreducible representation
  $V$. By block decomposition of $\C G$ and Schur's lemma it follows
  that $\pi$ restricts to a surjection from $A_V $ to the isotypical
  component of $V$ in $\C\mathcal C$.  However, since $V$ was defined
  over $\Q$ it follows that $A_V$ has a basis that lies inside $\Q
  G$. Moreover by surjectivity of $\pi|_{A_V}$ onto the isotypical
  component there exists such a basis element $b$ whose image $\pi(b)$
  is nonzero. By the assumptions, $b$ is an eigenvector of the Killing
  matrix. Moreover $b$ has rational coefficients as a vector in
  $\C\mathcal C$. Since the entries of $K$ are integral and $b$ is
  rational it follows that the eigenvalue of $b$ lies in $\Q$.  On the
  other hand the integrality of $K$ implies that the eigenvalues are
  all algebraic integers.  So the eigenvalue of $b$ is a rational
  number and an algebraic integer. Therefore it must lie in $\Z$.
\end{proof}

Note that this proposition implies in particular, that if $V$ is a
complex representation of $G$ defined over $\Q$ which occurs in
$\C\mathcal C$ with multiplicity $1$, then it lies in an eigenspace of
$K$ with eigenvalue in $\Z$.  Since all representations of the
symmetric group are defined over $\Q$ (over $\Z$ even), we have the
following corollary.

\begin{corollary}
  Let $\CC$ be a nontrivial conjugacy class of $S_n$. If an
  irreducible representation of $S_n$ occurs in the conjugation
  representation $\C\CC$ with multiplicity one, then it embeds into an
  eigenspace for the corresponding Killing form with eigenvalue in
  $\Z$. \qed \end{corollary}

Note that an irreducible representation is rational if all its
character values are rational. This is because the matrix entries can
be obtained by projection via central idempotents in the group algebra
with coefficients defined by the characters. Similarly the character
determines whether an irreducible representation is complex in the sense of
not real.

\begin{proposition} Let $\CC\subseteq G\setminus\{e\}$ be an ad-stable
  subset.
  \begin{enumerate}
  \item If a complex irreducible representation $V$ occurs in $\C\CC$
    in an eigenspace of the associated Killing form matrix, then so
    does its dual representation (with complex conjugate character).

  \item If we consider the inverse conjugacy class $\CC^{-1}$ then the
    eigenvalues of the Killing form matrix for $\CC^{-1}$ are the same
    as the ones obtained for $\CC$, and the decompositions of the
    respective eigenspaces into irreps are equivalent.
  \end{enumerate}
 
\end{proposition}
\begin{proof}
  The conjugation representation is
  clearly defined over $\R$, and since $K$ is real and symmetric in the basis $\CC$ its
  eigenspaces are also defined over $\R$, hence real as
  subrepresentations of the conjugation representation. This implies
  the first part. For the second part we consider inversion as a
  bijection between the two ad-stable subsets. Let
  $a,b,c\in\CC$. Clearly $c$ commutes with $ab$ precisely if
  $c^{-1}$ commutes with $b^{-1}a^{-1}$. But as the Killing forms are
  symmetric, we see that the Killing forms have the same matrices in
  their respective bases. If $v\in \C\CC$ is expanded in the basis
  $\CC$ we define $\tilde v$ to be the corresponding vector in
  $\C\CC^{-1}$ with the same coefficients in the corresponding basis,
  i.e. $v,\tilde v$ are represented by the same column vector in their
  respective bases. One may readily see that the matrices for the
  action of an element of $g$ in the two cases are also identical.
  This implies the second part.
\end{proof}


By a slight abuse of notation, in the following we denote the element 
$\sum_{a\in \C} a$ by $\theta$ (its analogue as a left-invariant 1-form makes the calculus 
inner). Clearly $\theta$ spans a copy of the trivial representation in $\C\CC$, and
the unique copy if $\CC$ is a conjugacy class. We also 
recall that a matrix with non-negative entries is called 
{\em irreducible} if for all indices $i,j$ there exists $m\in\N$ such that the matrix entry
$(K^m)_{ij}\ne 0$. This is equivalent to connectedness of the graph on
the set of indices defined by an edge whenever the entry $K_{ij}\ne
0$.

 \begin{proposition}\label{thetared} 
   Let $G$ be a finite group and $\CC\subseteq G\setminus\{e\}$ a
   conjugacy class. Then $K$ has a (positive) integral maximal
   eigenvalue $\lambda_{max}$, given by the sum of any column of
   $K$. Moreover, $K$ splits onto $r$ irreducible direct summands {\em if
     and only if} the eigenspace associated to $\lambda_{max}$ has
   dimension $r$ and in this case all other eigenspace dimensions are
   divisible by $r$. In particular, if $K$ is irreducible then the
   eigenspace associated to $\lambda_{max}$ is 1-dimensional, generated by the
   eigenvector $\theta=\sum_aa$.
 \end{proposition}

\begin{proof} 
  $K(\theta)=\sum_{a,b} K(a,b)b=\sum_b c_b b$ where $c_b$ is the sum
  of the $b$'th column of the matrix of $K$. However,
  $c_{gbg^{-1}}=\sum_aK(a,{gbg^{-1}})=\sum_a K({g^{-1}ag},b)=c_b$
  after a change of variables. Hence $c_b$ is independent of $b\in\CC$
  in the case of a conjugacy class. Hence $\theta$ is an eigenvector
  of $K$ with eigenvalue the column sum.  Moreover, if $K$ is
  irreducible then by Perron-Frobenius theory there is a 1-dimensional
  maximal eigenspace with eigenvalue the column sum of $K$, i.e. with
  eigenvector $\theta$. If $K$ is not irreducible then after a
  reordering of the basis it can be presented as a direct
  sum. Iterating this, we reduce $K$ to a direct sum of some number
  $r>1$ of irreducible blocks. In fact each block will be, after
  reordering, a copy of the same irreducible matrix. This follows from
  ad-invariance of $K$ as follows. Consider an element in $G$ that
  conjugates a corner of the first block to the corresponding corner
  of another. All the indices relating to the first block belong to
  the same connected component of the graph and, by assumption, they
  are not connected to any of the indices for the other blocks, and
  this notion is ad-invariant, as $K$ is. Hence the indices relating
  to the conjugated first block must be connected to themselves and
  not to the first block. Hence the first block maps over to the
  conjugated block, and all its entries are the same when suitably
  ordered, again by ad-invariance of $K$. Once $K$ has been presented
  as $r$ blocks $K_i$, its eigenvectors will consist of $r$ parts
  forming eigenvectors for each block with the same
  eigenvalue. However, since these blocks are all irreducible and have
  the same row sum as $K$, they will each have the same maximal
  eigenvalue as $K$, and any other eigenvalues will be strictly lower.
  This implies the facts stated and justified the notation
  $\lambda_{max}$ for the column sum.  Note that the diagonal of $K$
  is always nonzero as $a$ commutes with $a^2$ for all $a\in\CC$.
  Hence $K^{m+1}$ can only have the same or more positive entries as
  $K^m$, so in our case irreducible is equivalent to the existence of
  $m\in\N$ such that all entries of $K^m$ are positive, i.e.  to
  primitivity of the matrix $K$. \end{proof}

It appears for finite simple nonabelian groups that $K$ is irreducible
for every nontrivial conjugacy class $\CC$ not consisting of
involutions. This is surmised by looking at finite simple groups up to
order 75,000. The only observed reducible cases are the classes of
involutions fo $G=PSL(2,2^k)$, $G=PSU(3,2^k)$ or $G = Suz(2^{2k-1})$
for $k\geq 2$ up to the order that we could check. These are all
groups of Lie type over finite fields of characteristic
2. $S_4$ does have a noninvolutive reducible class (the 4-cycles) but for
$S_n$, $n>4$ we have checked by computer up to $n=8$
 that  the conjugacy classes with reducible $K$ are
  precisely the ${n-1\over 2}$-fold $2$-cycles for $n$ odd, so all are involutive. In this
  case the maximal eigenvalue has eigenspace decomposition $1\oplus
  (n-1)$, where $(n-1)$ means the standard representation.

\begin{lemma}\label{Cmunu} Let $G$ be a finite group and $\CGe$ a
  conjugacy class. Then $\mu,\nu,\rho$ in \eqref{munu} are defined and 
  \[  \mu={d\over
    \lambda_{max}}={1\over\<K\>},\quad \nu=\chi\mu,\quad
  \rho=\chi^2\mu,\quad 0<\lambda_{max}\le d^2,\quad  1\le \chi\le d\]
   where $\<K\>$ denotes the average entry of $K$,
  $d=|\CC|$ and $\chi=\chi_\CC(\CC)$ is the constant value of
  $\lambda^*(a)$ on $a\in \CC$. The upper bound for $\lambda_{max}$ holds iff $K$ has all entries $d$. 
\end{lemma}
\begin{proof} By Proposition~\ref{thetared} we know that $\theta$  is in the image of $K$ and that $\mu=\theta \cdot K^{-1}
  \theta={\theta\cdot\theta\over \lambda_{max}}={d\over
    \lambda_{max}}$. Since $\lambda=\chi\theta$ we then have $\nu,\rho$ as stated. Also since $\lambda_{max}$ is the column sum of $K$ it is clear that  $\lambda_{max}/d=\<K\>$. This is strictly positive since all entries are non-negative and $K(a,a)\ge 1$ for all $a\in \CC$.  The upper bound for $\lambda_{max}/d$ is saturated when $\<K\>=d$ which means every entry is $d$ as this is also the maximum of any entry.   \end{proof}

The upper bound for $\chi$ is reached precisely when all elements 
of $\CC$ mutually commute, which again
implies that all entries of $K$ are $d$, so apart from this case both upper bounds in the lemma are not reached. If the conjugacy class is real then $\lambda_{max}\ge d$ since for every $a\in \CC$ there exists $b\in \CC$ with $K(a,b)=d$. Meanwhile, $\chi\ge 2$ if the conjugacy class is real and not one of involutions. 

As regards nondegeneracy, we know from computer verification that all
finite simple nonabelian groups at least to order 75,000 and with real
conjugacy classes have nondegenerate $K$. In another direction we have
the following result:

\begin{proposition}\label{KCdisj} Let $G_1,\CC_1$ and $G_2,\CC_2$ be
  two finite groups with nontrivial conjugacy classes and $K_1,K_2$
  nondegenerate. Then $K_{\CC_1\sqcup\CC_2}$ is nondegenerate if and
  only if
  \begin{equation}\label{e:inequality} (\chi_1+\chi_2)^2\ne (\<K_1\>+d_2)(\<K_2\>+d_1)
  \end{equation} 
  where
  $\chi_i=\chi_{\CC_i}(\CC_i)$ and $d_i=|\CC_i|$.  Sufficient
  conditions for this are any of
  \begin{enumerate}
  \item $\chi_i< \<K_i\>$, $i=1,2$
  \item $\<K_1\>\<K_2\>\notin\Z$
  \item ${\rm max}\{\chi_1,\chi_2\}\le {\rm min}\{d_1,d_2\}$
  \item ${\rm max}\{\chi_1^2,\chi_2^2\}\le {d_1 d_2\over 2}$
  \end{enumerate}
\end{proposition}
\begin{proof} In the case of a conjugacy class the formula in
  Proposition~\ref{Kdis} becomes
  \[ K_{\CC_1\sqcup\CC_2}=\begin{pmatrix} K_1& 0 \cr 0&
    K_2\end{pmatrix}+\begin{pmatrix}d_2\theta_1^*\tens\theta_1^* \ \ &
    (\chi_1+\chi_2)\theta_1^*\tens \theta_2^*\cr
    (\chi_1+\chi_2)\theta_2^*\tens\theta_1^*&
    d_1\theta_2^*\tens\theta_2^*\end{pmatrix}\] as a bilinear form
  (similarly as a matrix). The formulae in Lemma~\ref{Cmunu} mean that
  the determinant condition reduces now to the one stated.  The listed
  sufficient conditions are immediate.  For (2) note that multiplying
  out the right hand side of the inequality \eqref{e:inequality} gives cross terms
  $\lambda_{max,i}$ which are integers. \end{proof}

Here (1) has the merit of being properties of each group and conjugacy
class separately and such groups and classes can be direct producted
with the direct sum `Lie algebra'. But it is not effective for simple
groups in the tables in the Appendix. Nevertheless we have the following Corollary. 

\begin{corollary} At least for finite simple nonabelian groups up to
  order 75,000 i.e. with reference to the tables in the Appendix,
  their direct product with the disjoint union of real conjugacy
  classes gives a nondegenerate Killing form.
\end{corollary}
\begin{proof} We apply test (4) in the preceding
  Proposition~\ref{KCdisj}. The largest value of
  $2\chi_\CC(\CC)^2/|\CC|$ in the tables is for the 2A class of $A_8$
  at about 11.9, when $\chi_\CC(\CC)$ is 25. This 11.9 is less than
  the smallest value of $|\CC|$ anywhere else in the tables, as the
  smallest size of a conjugacy class happens for classes $5A$ and $5B$
  in $A_5$, both with size 12. Moreover, any classes with
  $\chi_\CC(\CC)>25$ so as increase the left hand side have a much
  larger $|\CC|$ so that (4) still holds.
\end{proof}

\section{The Killing form and conjugation representations for
  $S_n$}\label{Sn}

Although Conjecture~\ref{con} and other points of discussion have been
for simple groups, the symmetric
groups are sufficiently close that we expect much of the discussion to apply to them
as well. Our main result, Proposition~\ref{KSn}, is for
$S_n$ with its 2-cycles class, namely an explicit decomposition of $\C \CC$
into irreducible representations in a manner compatible with the
eigenspace decomposition under $K$, and with explicit formulae for
the eigenvalues. In particular, we show that the Killing form
matrix $K$ for this conjugacy class is  nondegenerate. In this case
 it is necessarily positive definite by  Proposition~\ref{sig}. At the other extreme
 we find the maximal eigenvalue $\lambda_{max}$ for the $n$-cycles conjugacy class
 when $n$ is an odd prime.
 
First we note that in the case of $S_n$ for $n>4$, with the 2-cycles
conjugacy class, one can see from the formulae for the Killing form in
\cite{Ma:perm} that $K$ itself has all entries strictly
positive. Hence  Proposition~\ref{thetared} applies in this case
and there is a unique maximal eigenvalue, with eigenspace spanned by
$\theta$. For $S_3$ and $S_4$, $K$ is reducible, and $\theta$ is a
maximal eigenvector but each eigenvalue has multiplicity 3. 

We will need a concrete construction of irreducible subrepresentations
inside a conjugation representation.  For any partition
$\mu=(\mu_1,\dotsc, \mu_k)$ of $n$ we have a corresponding conjugacy
class $\mathcal C_\mu$ in $S_n$, namely the one with cycle type
$\mu$. Explicitly $\mathcal C_\mu$ is the conjugacy class containing
the element
\begin{equation}\label{e:amu}
a_{\mu}=(1,\,\dotsc\, ,\mu_1)(\mu_1+1,\,\dotsc\, ,\mu_1+\mu_2)\dotsc (
n-\mu_{k}+1,\,\dotsc\, ,n).
\end{equation}
If we let $Z_{a_\mu}$ denote the centraliser of $a_\mu$ and identify $
S_n/Z_{a_\mu}\cong \mathcal C_\mu$ via $\sigma Z_{a_\mu}\mapsto \sigma
a_\mu \sigma^{-1}$, then we obtain a $S_n$-equivariant homomorphism
from the left regular representation to the conjugation
representation,
\begin{equation}\label{e:pi}
\pi :\C S_n \longrightarrow \C (S_n/Z_{a_\mu})\cong \C \mathcal C_\mu,
\end{equation}
coming from the linear extension of the quotient map $S_n\to
S_n/Z_{a_\mu}$. If we interpret $\C S_n$ as the group algebra, then the map $\pi$
 becomes the action of $\C S_n$ on the element $a_\mu\in \C \mathcal C_\mu$. 
 The map $\pi$ is surjective reflecting $\C \mathcal C_\mu$ being a cyclic 
  $\C S_n$ module.

For the symmetric groups the irreducible representations are
very well known \cite{Fulton, Sagan, Fulton-Harris}, and  we have a concrete 
decomposition of $\C S_n$ into irreducibles
at our disposal.
Namely, recall that irreducible representations $S^\lambda$ of $S_n$ are indexed by partitions 
$\lambda \vdash n$, and a partition is represented by its {\it Young diagram} or {\it
  shape}.
Since
$S^\lambda$ occurs in $\C S_n$ with multiplicity
equal to $\dim S^\lambda$, the construction of a subrepresentation
of $\C S_n$ isomorphic to $S^\lambda$ for given $\lambda$
must naturally depend on an additional choice, so choose a {\it
  tableau} of shape $\lambda$, a one-to-one labelling of the boxes by
the integers $\{1,\dotsc, n\}$. The symmetric group $S_n$
acts on the set of tableaux by permuting the entries, and 
therefore a tableau $T$ defines a subgroup $R(T)$ of permutations
preserving the row sets, and a subgroup $C(T)$ of permutations preserving
the column sets.
The corresponding irreducible summand in $\C S_n$ is the submodule 
$S^T:=\C S_n \, c_T$, which is generated by the 
`Young symmetrizer'  $c_T=b_T a_T $ of $T$,  where
$$
a_T=\sum_{\sigma\in R(T)}\sigma, \qquad b_T=\sum_{\sigma\in
  C(T)}{\epsilon(\sigma)} \sigma,
$$
with $\epsilon(\sigma)$ the sign of the permutation $\sigma$.


 Clearly the right action of $S_n$ on $\C S_n$ provides
$S_n$-equivariant isomorphisms between the modules $S^T$ for varying
$T$ making them all equivalent. Note that there are many more tableaux than the multiplicity of
$S^\lambda$.  Let $SYT(\lambda)$ denote
the set of {\it standard Young tableaux}, that is tableaux whose
entries are strictly increasing in rows and in columns. Then the isotypic component
of $S^\lambda$ inside $\C S_n$ is precisely the subspace
$$
\bigoplus_{T\in SYT(\lambda)} S^T.
$$
It is now straightforward to find the
irreducible summands of $\C \CC_\mu$  using Young symmetrizers, 
as follows.

\begin{lemma}\label{l:Spechtembeddings}
  Suppose $\lambda$ and $\mu$ are partitions of $n$ and all notations are as above.

  The Specht module $S^\lambda$ occurs as a subrepresentation of the
  conjugation representation $\C \mathcal C_\mu$ if and only if there
  exists a standard Young tableau $T$ of shape $\lambda$ for which
  $c_T\cdot a_{\mu}\ne 0$ in $\C \mathcal C_\mu$.

  In that case, the subrepresentation is explicitly realized as the
  subspace $\pi(S^T)$, where $\pi$ is the projection map from
  \eqref{e:pi}.
\end{lemma}

\begin{proof}
  If there is a tableau $T$ for which $c_T\cdot a_{\mu}\ne 0$, then
  the restriction of the map $\pi$ from \eqref{e:pi} to the
  subrepresentation $S^T$ of $\C S_n$ defines a nonzero
  $S_n$-equivariant map $S^T\to \mathcal C_{\mu}$. Since $S^T$ is
  irreducible and isomorphic to $S^\lambda$ it follows that this map
  must be an isomorphism onto its image.

  On the other hand, if $c_T\cdot a_{\mu}= 0$ for all $T\in
  SYT(\lambda)$, then the entire block of $S^\lambda$ in $\C S_n$ lies
  in the kernel of $\pi$, and therefore the irreducible representation
  $S^\lambda$ does not occur in the image of $\pi$.  Since $\pi$ is
  surjective, this means that $S^\lambda$ is not a subrepresentation
  of $\C \mathcal C_\mu$.
\end{proof}

We remark that this lemma also holds, of course, with $\C \CC_\mu$
replaced by any cyclic $\C S_n$-module.

\subsection{$S_n$ with the 2-cycles class}

  In the example of $S_3$, the $2$-cycles class $\CC_{(2,1)}$
has three elements and it is straightforward to see that the
conjugation representation, $\C \CC_{(2,1)}$, is the (defining)
three-dimensional permutation representation of $S_3$.  In terms of
Specht modules this representation decomposes as
\begin{equation} \label{e:2cyclesS3} \C \CC_{(2,1)}=S^{(3)}\oplus
  S^{(2,1)}.
 \end{equation}
 That is, the trivial representation plus the standard $2$-dimensional
 representation.  The general case is not much different. We will use the notation $(2,1^{n-2})$ for the partition
 $(2,1,\overset{(n-2)}{\dotsc}, 1)$ which represents the $2$-cycles
 class in $S_n$.

 \begin{proposition}\label{KSn} Consider $S_n$ for $n>2$ with the 2-cycles class
   $\CC=\CC_{(2,1^{n-2})}$. For $n=3$ the decomposition of $\C \CC$
   into irreducibles is given in equation \eqref{e:2cyclesS3}.
   \begin{enumerate}
   \item For $n>3$ the decomposition of the conjugation representation
     $\C \CC$ into irreducible representations is given by
     \[
     \C \CC\cong S^{(n)} \oplus S^{(n-1,1)}\oplus S^{(n-2,2)}.
     \]
     Here the first two Specht modules $S^{(n)},S^{(n-1,1)}$ are the
     trivial representation and the standard $(n-1)$-dimensional
     representation, respectively.
   \item Each irreducible submodule of $\C \CC$ lies in an eigenspace
     for the Killing form matrix $K$ with eigenvalues as follows.  The
     eigenvalue of $K$ for the eigenspace containing $S^{(n)}$
     (spanned by the element $\theta$) is
     \[ {1\over 4}(n^4-10 n^3+41 n^2-72 n+48).
     \]
     The eigenvalue of $K$ in the eigenspace containing $S^{(n-1,1)}$
     is
     \[ n^2-6n+12.
     \]
     Suppose $n>3$.  Then the eigenvalue of $K$ on the eigenspace
     containing $S^{(n-2,2)}$ is $2n$.
   \end{enumerate}
 \end{proposition}
 \begin{proof}
   Part (1) could be checked using character theory. But we will
   rather define explicit embeddings of the Specht modules, by the
   method of Lemma~\ref{l:Spechtembeddings}, in order to be able to
   compute the eigenvalues of $K$ in the later parts of the proof.

   Of course the trivial representation embeds into $\C \CC_{(2,
     1^{n-2})} $ as the subspace spanned by the element
   $\theta=\sum_{a\in\CC}a$, and has multiplicity $1$.

   For the standard representation $S^{(n-1,1)}$ we consider the
   subspace $\pi(S^{T_1})$ of $\C \CC_{(2, 1^{n-2})}$ for $\pi$ from
   \eqref{e:pi} corresponding to the tableau
   \[
   T_1 =
   \begin{ytableau}
     1 & 2 & 3 & \dotsm & \scriptstyle n-1 \\
     n
   \end{ytableau}
   \]

   This is the submodule of $\C \CC$ obtained by applying $ \C S_n$ to
   the vector $c_{T_1}\cdot (12)$. Up to an overall multiple, which we
   drop, this vector works out to be
   $$
   v_{T_1}=(12)+(13)+\cdots+(1,n-1)-(2,n)-(3,n)-\cdots-(n-1,n).
   $$
   Since $v_{T_1}\ne 0$ we have found a copy of $S^{(n-1,1)}$ in $\C
   \CC$.

   For the next representation $S^{(n-2,2)}$ we consider the subspace
   $\pi(S^{T_2})$ of $\C \CC$ for $\pi$ from \eqref{e:pi} and the
   tableau
   \[
   T_2 =
   \begin{ytableau}
     1 & 2 & 3 & \dotsm & \scriptstyle n-2 \\
     \scriptstyle n-1 & n
   \end{ytableau}
   \]
   This is the submodule of $\C \CC$ obtained by applying $ \C S_n$ to
   the vector $c_{T_2}\cdot (12)$. Up to an overall multiple this
   vector works out to be
   \[
   v_{T_2}=(12)-(2,n-1)-(1,n)+(n-1,n),
   \]
   and since $v_{T_2}\ne 0$ we have found a copy of $S^{(n-2,2)}$ in
   $\C \CC$.

   That we have thereby completely decomposed $\C \CC$ follows by
   dimension count:
   \[
   \dim{S^{(n)}}+\dim{ S^{(n-1,1)}} +\dim{S^{(n-2,2)}}= 1+(n-1)+
   \frac{n(n-3)}2 ={n\choose 2}=\dim{\C \CC},
   \]
   where $\dim(S^{(n-2,2)})$ is computed for example by the hook
   formula. 
    This concludes the proof of (1).

   That the irreducible subrepresentations lie in eigenspaces of $K$
   follows immediately from the fact that in the decomposition of
   $\C\CC$ each irreducible representation occurs with multiplicity at
   most one.  We can now compute the eigenvalues.

   For the trivial representation we compute the column sum $\sum_a
   K((12),a)$ over all 2-cycles. In the basis of the `triangular'
   listing
   \[ (12)\]
   \[(13),(23)\]
   \[(14),(24),(34)\]
   \[ (15),(25),(35),(45)\]
   \[ \vdots\quad\vdots\]
   \[(1n),(2n),(3n),(4n),\cdots ,(n-1,n)\] we have for $a$ the choice
   (12), or $a$ lies in the size $2(n-3)$ region on the left where $a$
   has one entry in common with (12), or $a$ lies in the triangle to
   the right of size $(n-2)(n-3)/2$ where $a$ is disjoint from
   (12). Using the values of $K$ for these three cases in
   \cite{Ma:perm}, we find
   \[ {n\choose 2}+2(n-2){n-3\choose 2}+{(n-2)(n-3)\over
     2}\left({n-4\choose 2}+2\right)\] which computes as stated.

   For the standard representation we use the vector we constructed in
   the proof of (1),
   \[ v_{T_1}=(12)+(13)+\cdots+(1,n-1)-(2,n)-(3,n)-\cdots-(n-1,n),\]
   which involves the left and bottom slopes of the triangle leaving
   out the common vertex.  Then the eigenvalue computed as the
   coefficient of $(12)$ in $K(v_{T_1})$ is
   \[
   K((12),(12))+(n-3)K((12),(13))-K((12),(2,n))-(n-3)K((12),(3,n))\]
   \[ = {n\choose 2}+(n-4){n-3\choose 2}+(n-3)\left({n-4\choose
       2}+2\right)\] which comes out as stated. Both formulae,
   although computed for $n>4$ in the above counting, also give the
   right answer for $n=2,3,4$, as computed by hand.
   
   For the representation $S^{(n-2,2)}$ we use the vector
   \[ v_{T_2}=(12)-(2,n-1)-(1,n)+(n-1,n)\] from the proof of (1) and
   compute the eigenvalue as the $(12)$ coefficient of $K(v_{T_2})$,
   i.e. as
   \[ K((12),(12))-K((12),(2,n-1))-K((12),(1,n))+K((12),(n-1,n))\]
   \[ = {n\choose 2}-2{n-3\choose 2}+{n-4\choose 2}+2=2n.\]
 \end{proof}

 \begin{corollary}\label{Sn2cy} The Killing form for $S_n$, $n>2$ with
   the 2-cycles conjugacy class $\CC$ is non-degenerate and in fact
   positive definite. Moreover the decomposition of $\C \CC$ into
   irreps consisting of the trivial and the standard representation,
   and the representation $S^{(n-2,2)}$, coincides for $n>6$ with the
   decomposition of $K$ into eigenspaces of respectively the maximal,
   next to maximal and smallest eigenvalues.
 \end{corollary}
 \begin{proof} Looking at the three expressions for the eigenvalues in
   the lemmas above it is evident that they have different leading
   powers of $n$ and hence are distinct for all $n$ bigger than some
   value. By inspection, the only degeneracies are $n=3$ when the
   trivial and the standard representation have the same eigenvalue of
   $K$, $n=4$ when the eigenvalues of the trivial and the
   $S^{(n-2,2)}$ coincide, being smaller than the eigenvalue of the
   standard representation, and $n=6$ when the eigenvalues of the
   standard representation and of $S^{(n-2,2)}$ coincide. After that,
   the eigenvalue of the trivial exceeds that of the standard
   representation which exceeds that of $S^{(n-2,2)}$ as stated. As
   all the eigenvalues are positive we conclude that $K$ is
   non-degenerate (and positive definite when extended as a hermitian
   inner product). \end{proof}

\subsection{$S_n$ with the $n$-cycles class}\label{signrep}

In this section $\CC$ is the class of $n$-cycles and our first result is a formula for the eigenvalue $\lambda_{max}$ of the
Killing form, whose eigenspace the trivial representation in $\C \CC$ for $n$ prime, using a result of 
Zagier\cite{Zagier}.

\begin{proposition}\label{p:lmax} Let $n$ be an odd prime. The maximal eigenvalue  of the Killing form on $S_n$ with
its $n$-cycles class is $\lambda_{max}= \frac{(n-1)!}{n+1}(3n-1)$. 
\end{proposition}
\proof Suppose $a$ and $b$ are $n$-cycles for which
the product $ab$ is an $n$-cycle. The centraliser of $ab$ consists in this case of all the powers of $ab$. Since $n$ is prime, these powers are all $n$-cycles except for the $n$-th power which is $e$. So in this case $K_{a,b}=|Z(ab)\cap \CC|=n-1$. 
 If $ab$ is not an $n$-cycle or the identity, it cannot commute with an $n$-cycle, and hence $K_{a,b}=0$ in that case. 
Finally in the case where $ab=e$ we have $K_{a,b}=| \CC|=(n-1)!$.  
By a result of Zagier's, \cite{Zagier}, it is known that for each fixed $n$-cycle $a$, there are $\frac{2(n-1)!}{n+1}$ many $n$-cycles $b$ such that $ab$ is again an $n$-cylce. 
Hence the eigenvalue $\lambda_{max}$ which is the row sum of $K$ is $\frac{2(n-1)! }{n+1}(n-1)+(n-1)!$. This simplifies to the formula in the proposition. \endproof

Next we look at the sign representation. As a small digression
we first establish precisely which conjugacy class this occurs in. In particular 
it occurs in the $n$-cycles class precisely when $n$ is odd, and we will
conjecture a generalisation of the Proposition~\ref{p:lmax}. 

Recall that the overall multiplicity of the sign
representation in the conjugation representation $\C S_n$ is easily found
by character theory as precisely the number of
conjugacy classes consisting of even permutations minus the number of
conjugacy classes of odd permutations (the row sum in the character
table, for the sign representation). If $s(n)$ denotes the
multiplicity of the sign representation in $\C S_n$, then the above
description of $s(n)$ implies the product formula 
\begin{equation}\label{e:prod1}
1+t+\sum_{n=2}^{\infty}s(n)t^n=\prod_{k=1}^{\infty}\left(\frac{1}{1+(-t)^k}\right).
\end{equation}

By a classical Euler identity which reads (after replacing the usual
variable by $-t$ and inverting),
\begin{equation}\label{e:prod2}
\prod_{k=1}^{\infty}\left(\frac{1}{1+(-t)^k}\right)=\prod_{k=1}^{\infty}(1+t^{2k-1}),
\end{equation}
it follows that the multiplicity of the sign representation in the
conjugation representation $\C S_n$ is equal to the number of
partitions of $n$ into distinct odd parts. The following is surely also known but
we have not found a reference and include it here. 

\begin{proposition}\label{p:signrep} The sign representation of $S_n$ appears as a subrepresentation of 
  the conjugation representation $\C\CC_\mu$ if and only if $\mu$ is a
  partition of $n$ into distinct odd parts. If it appears in
  $\C\CC_\mu$, then it has multiplicity one.
\end{proposition}

\begin{proof}
  Since the sign representation has multiplicity one in in the
  left-regular representation $\C S_n$ and the conjugation
  representation $\C\CC_\mu$ is a cyclic $\C S_n$-module, it is clear
  that the sign representation can have multiplicity at most $1$ in
  $\C\CC_\mu$.

  Let us now write $\sigma\cdot \sigma'=\sigma\sigma'\sigma^{-1}$ for
  the conjugation action.  Fix an element $a_\mu$ in the conjugacy
  class $\CC_\mu$.  By Lemma~\ref{l:Spechtembeddings}, the sign
  representation appears in $\C \CC_{\mu}$ if and only if the element
  $$
  v_\mu=\sum_{\sigma}{\epsilon(\sigma)} \sigma \cdot a_\mu
  $$
  in $\C \CC_{\mu}$ is nonzero. Moreover if it is nonzero then it
  spans the sign representation.  Now suppose $v_\mu$ is nonzero and
  let $\tau$ be an element of the centraliser $Z_{a_\mu}$.  Then we
  see that
  \begin{equation*}
    \tau \cdot v_\mu = \sum_{\sigma}{\epsilon(\sigma)}(
    \tau\sigma\tau^{-1}) \tau \cdot a_\mu
    =\sum_{\sigma}{\epsilon(\sigma)} \tau\sigma\tau^{-1} \cdot a_\mu=
    v_\mu.
  \end{equation*}
  This implies that $\tau$ is even, since $v_\mu$ spans the sign
  representation.  Therefore if the sign representation occurs in $\C
  \CC_\mu$ then $Z_{a_\mu}$ contains only even permutations.

  The converse is true as well. If all elements in $Z_\mu$ are even,
  then the coefficient of $a_\mu$ in $v_\mu$ comes out to be $| Z_\mu
  |$, implying that $v_\mu$ is nonzero, and the sign representation
  occurs in $\C C_\mu$.

  It remains to prove that $Z_{a_\mu}$ contains only even
  permutations, precisely if $\mu$ is a permutation of $n$ into
  distinct odd parts.

  Clearly, if $\mu$ has an even part then there is a cycle of even
  length in $a_\mu$, which gives an element of the centralizer that
  has odd parity. Also if $\mu$ has two parts of size $k$ (we may
  assume $k$ odd, by above), then there is an element of the
  centralizer which exchanges the corresponding two $k$-cycles of
  $a_\mu$, which is a product of $k$ many $2$-cycles. So again there
  is an element of odd parity in $Z_{a_\mu}$.  This shows that for the
  sign representation to occur inside $\C \CC_\mu$, we must have that
  $\mu$ is a partition of $n$ into distinct, odd parts.
  
  Conversely, if $\mu$ is a partition of $n$ into distinct odd parts,
  then the centraliser is generated by the individual cycles in
  $a_\mu$, and these are all even permutations.
\end{proof}

\begin{remark} 
  Another well-known partition identity gains a
  representation-theoretic interpretation in this context.  Namely the
  block decomposition of $\C S_n$ is also invariant under the
  conjugation representation, and it is easy to check using character
  theory that the sign representation occurs, and with multiplicity
  one, precisely in the blocks of Specht modules corresponding to
  transpose-symmetric partitions.  This gives another explanation of
  the fact that the number of transpose-symmetric partitions of $n$
  agrees with the number of partitions of $n$ into distinct odd parts
  (a fact which has an easy, not obviously related bijective proof) .
\end{remark}

Proposition~\ref{p:signrep} implies,
as mentioned before, that  the sign representation occurs in the class of $n$-cycles
iff $n$ is odd. In this case  we 
define the parity $\pi(a)$ of $a\in \CC$ to be  $\pi(a)=\epsilon(\sigma)$ where $\sigma$ is
any permutation for which
$\sigma a_\mu \sigma^{-1}=a$. Here $\pi$ is well-defined as any permutation that
commutes with an $n$-cycle has to be a power of an $n$-cycle and hence even, as $n$ is odd. We then let
\[ \delta_n=\#\{a\in \CC\ |\ a_\mu a\in \CC, \ \pi(a)=1\}- \#\{a\in \CC\ |\  a_\mu a\in \CC, \ \pi(a)=-1\}\]
where $a_\mu=(1,\dotsc, n)$. This  $\delta_n$ is a signed version of 
 our previous $\#\{a\in \CC\ |\ a_\mu a\in \CC\}$. 

\begin{lemma} Let $\CC$ be the class of $n$-cycles  in $S_n$ with $n$ an odd prime. The
eigenvalue of $K$ on the eigenspace in $\C\CC$ containing the sign representation is given by
\[  \lambda_{sign}=\delta_n \, (n-1)+(-1)^{{n}\choose 2}(n-1)! .\]
\end{lemma}
\proof 
The vector 
$$
v={1\over n}\sum_{\sigma\in S_n} {\epsilon(\sigma)}\sigma a_\mu \sigma^{-1}\\
= \sum_{b\in\CC} \pi(b) b,
$$
spans the sign representation, and contains $a_\mu$ with coefficient $1$. Applying $K$ gives 
$$
K\left(\sum_{b\in\CC} \pi(b) b\right)=\sum_{a,b\in\CC}\pi(b) |Z(a b)\cap\CC | a,
$$
and the desired eigenvalue is the new coefficient of $a_\mu$,
$$
\lambda_{sign}=\sum_{b\in\CC}\pi(b) |Z(a_\mu b)\cap\CC |.
$$
Since $n$ is prime, 
as in the proof of Proposition~\ref{p:lmax}, an $n$-cycle can lie in $Z(a_\mu b)$ only if
either $a_\mu b$ is itself an $n$-cycle, or if $a_\mu b=e$. 
Moreover, the cardinality $|Z(a b)\cap\CC |$ is $(n-1)$ in the first case, 
respectively $(n-1)!$ in the second case.
If follows that 
$$
\lambda_{sign}=\delta_n\, (n-1) + \pi(a_\mu^{-1})\,  (n-1)!.
$$
Clearly $\pi(a_\mu^{-1})$ is the sign of the longest permutation,
which is $(-1)^{n\choose 2}$ and the formula follows.
 \endproof

Finding the  $\delta_n$ would seem to require a refinement of 
Zagier's formula \cite{Zagier} for $\#\{a\in \CC\ |\ a_\mu a\in \CC\}$ 
into a sum of `odd and even' parts. However we conjecture for all  odd $n$ that 
\[ \delta_{n}=({n-1\over 2})!^2,\]
which we have verified for odd $n\le 9$. This conjecture if true 
would imply that for $n=2m+1$ an odd prime, 
\[ \lambda_{sign}= m!^2(2m)+(-1)^{m}(2m)!.\]

\appendix
\section{Computer verifications for simple groups}

To provide evidence for our conjectures and get a grip on the
behaviour of the Killing forms associated to minimal calculi for
finite simple groups we have performed an extensive amount of
computational verifications using the open source computer algebra
systems Sage and GAP. Code is available from the authors upon request.
In the present section we summarize our methods and results. Naming of
the conjugacy classes follows convention in the Atlas of finite simple
groups \cite{ATLAS}.

\subsection{Effective calculation of the Killing form}

To compute the Killing form $K$ associated to a conjugacy class $\CC =
g^G$ we take advantage of the ad-invariance $K({aga^{-1}}, {h}) = K(g,
{a^{-1}ha})$ by computing a section $s:\CC\to G$ satisfying $h =
s(h)gs(h)^{-1}$ for all $h\in \CC$, and using $K_{ab}=
\lambda_g^*(s(a)^{-1}bs(a))$, where $\lambda_g^*(h) := \lambda^*(gh) =
|Z(gh)\cap \CC|$, reducing computation of the Killing form to
computing its first row and the permutations that create the remaining
rows from that one.  Current limiting factor of the implementation is
computer memory, the first conjugacy class out of our reach is the
class 6B of elements of order 6 with centralizer of size 6 in the
Mathieu group $M_{12}$ of order 95040.

\subsection{Nondegeneracy}

Most of the simple groups are nondegenerate because they are Roth. The only non-Roth groups up to order 75000 are (cf. \cite{HZ}) $PSU(3,3)$ and $PSU(3,4)$. Of these, a direct computation shows that $PSU(3,3)$ is \textbf{not} nondegenerate. $PSU(3,4)$ is too large for direct inspection.

Nondegeneracy of the Killing forms for conjugacy classes is checked directly by computing its rank.  Up to
order 75,000 the Killing form is nondegenerate in all cases except
conjugacy classes 7A and 7B of elements of order 7 in the alternating
group $A_7$, conjugacy classes 4A, 4B, 8A, 8B, 12A, 12B of elements of
orders 4, 8 or 12 in the unitary group $PSU(3,3) = G_2(2)'$, and
conjugacy classes 7A and 7B of elements of order 7 in $PSL(3,4)$.  All
the degenerate cases occur in conjugacy classes that are not closed
under inversion, with \emph{real} conjugacy classes yielding
nondegenerate Killing forms.

\subsection{Irreducibility}
The irreducibility of $K$ is tested by checking connectedness of the
graph $G_K$ with vertices indexed by elements of $\CC$ and containing
an edge $(a,b)$ if and only if $K_{a,b}\neq 0$. This test does not
produce any noticeable overhead.

Generically, the tested Killing forms are irreducible, so
Perron-Frobenius theorem applies and the eigenspace associated to the
maximal eigenvalue is 1-dimensional; the only observed exceptions are
given by the conjugacy classes of involutions in the linear groups
$PSL(2, 4)=A_5$, $PSL(2, 8)$, $PSL(2, 16)$, $PSL(2,32)$, the
exceptional Suzuki group $Suz(8)$ and the unitary group $PSU(3,4)$,
all groups of matrices with coefficients on a field of characteristic
2. However, not every such group and class of involutions is
reducible, as shown by the data for $PSL(3, 4)$.

\subsection{Eigenspaces and irrep decompositions}
The computation of the characteristic polynomial and the eigenvalues
gets very slow as the size of the conjugacy classes
increase. Eigenvalues (with multiplicity) have been computed for all
the listed groups, revealing that the Killing form appears to be
positive definite whenever it comes from a conjugacy class consisting
of involutions plus the (non real) classes 3A and 3B of elements of
order 3 and centralizer of size 648 in the unitary group $PSU(4,2)$.
The link between involutions and positive definite Killing forms is
made clear for the nondegenerate case in Proposition \ref{sig}, with
the data showing that neither nondegeneracy nor being closed under
inversion can be relaxed.

For the same groups, we have also computed the decomposition into
irreps of the adjoint representation on $\C\CC$ by means of character
theory, looking for some correlation between both decompositions. As
the groups get larger, the observed behaviour is that the dimensions of
the eigenspaces coincide with the dimensions of irreducible
representations, so as the group size increases we expect that each
eigenspace contains exactly one irrep. The obvious exceptions to this
rule are the conjugacy classes yielding reducible Killing forms
mentioned in the previous paragraph.


\subsection{Data:} Here we summarize some of the obtained data for all
finite simple groups up to order 75,000. We list whether the conjugacy
class is real, reducibility of the Killing form, and its
signature. The naming of the conjugacy classes follows the convention
at the Atlas, and conjugacy classes of elements order with the same
centralizer sizes have been amalgamated whenever they show identical
behaviour. Listing the actual eigenspace decomposition of the adjoint
representation on $\C\CC$ would be too lengthy and not particularly
enlightening, so we shall omit that data here. Whenever the Killing
form is reducible we have included in the corresponding column the
number of irreducible components. Signature is expressed as $(p,n,z)$
where $p$, $n$ and $z$ are respectively the number (counted with
multiplicities) of positive, negative and zero eigenvalues; in
particular, nondegeneracy is given by zero as the last number of this
triple. In supplementary information we list the maximal eigenvalue
$\lambda_{max}$ of the Killing form, equal to the row sum. For a real
conjugacy class $(\lambda_{max}-|\CC|)/|\CC|$ is a measure of the
typical size of the other entries of the Killing form matrix after the
principal entry $|\CC|$ in each row.  We also list the value
$\chi_\CC(\CC)$ of the character of the adjoint representation on a
typical element of $\CC$ as a measure of the degree to which the
braided Lie algebra is nonabelian. It counts the number of elements in
$\CC$ that commute with any given element of $\CC$.

$A_5$, order 60\hfill\\
\begin{tabular}{|c|c|c|c|c|c|c|}
  \hline
  $\CC$ & $|\CC|$ & $\chi_\CC(\CC)$ & Real & Irred & $\lambda_{\text{max}}$ & Signature    \\\hline
  $2A$ & 15   &  $3$ &{\rm True} & False (5)& $21$  & (15, 0, 0)   \\\hline
  $3A$ & 20   &  $2$ &{\rm True} & True     & $34$  & (10, 10, 0)  \\\hline
  $5A-B$ & 12 &  $2$ &{\rm True} & True     & $24$  & (6, 6, 0)   \\\hline
\end{tabular}

$PSL(2,7)$, order 168\hfill\\
\begin{tabular}{|c|c|c|c|c|c|c|}
  \hline
  $\CC$ & $|\CC|$ & $\chi_\CC(\CC)$ & Real & Irred & $\lambda_{\text{max}}$ & Signature \\\hline
  $2A$ & 21  & $5$ &{\rm True} & {\rm True}   & $49$ &  (21, 0, 0) \\\hline
  $3A$ & 56 &  $2$ &{\rm True} & {\rm True}   & $94$ &   (28, 28, 0) \\\hline
  $4A$ & 42 &  $2$ &{\rm True} & {\rm True}   & $76$ &   (21, 21, 0) \\\hline
  $8A-B$ & 24 &$3$ &  {\rm False} & {\rm True}& $30$ & (16, 8, 0) \\\hline
\end{tabular}

$A_6$, order 360\hfill\\
\begin{tabular}{|c|c|c|c|c|c|c|}
  \hline
  $\CC$ & $|\CC|$ & $\chi_\CC(\CC)$ & Real & Irred & $\lambda_{\text{max}}$ & Signature \\
  \hline
  $2A$ & 45   & $5$ &  {\rm True}& {\rm True}	 & $73$	& (45, 0, 0)\\\hline
  $3A-B$ & 40 & $4$ &  {\rm True}& {\rm True}	 & $88$ & (20, 20, 0)	\\\hline
  $4A$ & 90   & $2$ &  {\rm True}& {\rm True}	 & $156$	& (45, 45, 0)\\\hline
  $5A-B$ & 72 & $2$ &  {\rm True}& {\rm True}	 & $134$& (36, 36, 0)	\\\hline
\end{tabular}

$PSL(2,8)$, order 504\hfill\\
\begin{tabular}{|c|c|c|c|c|c|c|}
  \hline
  $\CC$ & $|\CC|$ & $\chi_\CC(\CC)$ & Real & Irred & $\lambda_{\text{max}}$ & Signature \\
  \hline
  $2A$ & 63   & $7$ & {\rm True} & {\rm False} (9) & $105$& (63, 0, 0)\\\hline
  $3A$ & 56   & $2$ & {\rm True} & {\rm True}	 & $112$& (28, 28, 0)\\\hline
  $7A-C$ & 72 & $2$ & {\rm True} & {\rm True}	 & $130$& (36, 36, 0)	\\\hline
  $9A-C$ & 56 & $2$ & {\rm True} & {\rm True}	 & $112$& (28, 28, 0)	\\\hline
\end{tabular}

$PSL(2,11)$, order 660\hfill\\
\begin{tabular}{|c|c|c|c|c|c|c|}
  \hline
  $\CC$ & $|\CC|$ & $\chi_\CC(\CC)$ & Real & Irred & $\lambda_{\text{max}}$ & Signature \\
  \hline
  $2A$ & 55    & $7$ & {\rm True} & {\rm True} & $121$ & (55, 0, 0) \\\hline
  $3A$ & 110   & $2$ &{\rm True} & {\rm True}  & $208$ & (55, 55, 0) \\\hline
  $5A-B$ & 132 & $2$ &{\rm True} & {\rm True}  & $234$ & (66, 66, 0) \\\hline
  $6A$ & 110   & $2$ &{\rm True} & {\rm True}  & $208$ & (55, 55, 0) \\\hline
  $11A-B$ & 60 & $5$ &{\rm False} & {\rm True} & $80$ & (36, 24, 0) \\\hline
\end{tabular}

$PSL(2,13)$, order 1092\hfill\\
\begin{tabular}{|c|c|c|c|c|c|c|}
  \hline
  $\CC$ & $|\CC|$ & $\chi_\CC(\CC)$ & Real & Irred & $\lambda_{\text{max}}$ & Signature \\
  \hline
  $2A$ & 91    & $7$ &  {\rm True} & {\rm True}  & $157$ & (91, 0, 0) \\\hline
  $3A$ & 182   & $2$ &  {\rm True}& {\rm True}   & $328$ & (91, 91, 0) \\\hline
  $6A$ & 182   & $2$ &  {\rm True}& {\rm True}   & $328$ & (91, 91, 0) \\\hline
  $7A-C$ & 156 & $2$ &  {\rm True}& {\rm True}   & $298$ & (78, 78, 0) \\\hline
  $13A-B$ & 84 & $6$ &  {\rm True}& {\rm True}   & $192$ & (42, 42, 0) \\\hline
\end{tabular}

$PSL(2,17)$, order 2448\hfill\\
\begin{tabular}{|c|c|c|c|c|c|c|}
  \hline
  $\CC$ & $|\CC|$  & $\chi_\CC(\CC)$ & Real & Irred & $\lambda_{\text{max}}$ & Signature\\
  \hline
  $2A$ & 153   & $9$ &  {\rm True} & {\rm True}  & $273$ & (153, 0, 0)  \\\hline
  $3A$ & 272   & $2$ &  {\rm True} & {\rm True}  & $526$ & (136, 136, 0)  \\\hline
  $4A$ & 306   & $2$ &  {\rm True} & {\rm True}  & $564$ & (153, 153, 0)  \\\hline
  $8A-B$ & 306 & $2$ &  {\rm True} & {\rm True}  & $564$ & (153, 153, 0)  \\\hline
  $9A-C$ & 272 & $2$ &  {\rm True} & {\rm True}  & $526$ & (136, 136, 0)  \\\hline
  $17A-B$ & 144&  $8$ &{\rm True} & {\rm True}    & $336$ & (72, 72, 0)  \\\hline
\end{tabular}

$A_7$, order 2520\hfill\\
\begin{tabular}{|c|c|c|c|c|c|c|}
  \hline
  $\CC$ & $|\CC|$ & $\chi_\CC(\CC)$ & Real & Irred & $\lambda_{\text{max}}$ & Signature \\
  \hline
  $2A$ & 105 &  $9$ & {\rm True} & {\rm True}   & $273$& (105, 0, 0)\\\hline
  $3A$ & 70  & $10$ & {\rm True} &  {\rm True}  & $256$& (35, 35, 0)\\\hline
  $3B$ & 280 &  $4$ & {\rm True} &  {\rm True}  & $616$& (140, 140, 0)\\\hline
  $4A$ & 630 &  $2$ & {\rm True} &  {\rm True}  & $1068$& (315, 315, 0)\\\hline
  $5A$ & 504 &  $4$ & {\rm True} &  {\rm True}  & $936$& (252, 252, 0)\\\hline
  $6A$ & 210 &  $6$ & {\rm True} &  {\rm True}  & $528$& (105, 105, 0)\\\hline
  $7A-B$ & 360& $3$ &  {\rm False} & {\rm True} & $324$& (171, 140, 49)  \\\hline
\end{tabular}

$PSL(2,19)$, order 3420\hfill\\
\begin{tabular}{|c|c|c|c|c|c|c|}
  \hline
  $\CC$ & $|\CC|$ & $\chi_\CC(\CC)$ & Real & Irred & $\lambda_{\text{max}}$ & Signature \\
  \hline
  $2A$  & 171   &  $11$ &{\rm True}  & {\rm True} & $361$  & (171, 0, 0)\\\hline
  $3A$  & 380   &  $2$ & {\rm True}  & {\rm True}& $706$  & (190, 190, 0)\\\hline
  $5A-B$  & 342 &  $2$ & {\rm True}  & {\rm True}& $664$& (171, 171, 0)\\\hline
  $9A-C$  & 380 &  $2$ & {\rm True}  & {\rm True}& $664$& (190, 190, 0)\\\hline
  $10A-B$ & 342 &  $2$ & {\rm True}  & {\rm True}& $706$& (171, 171, 0)\\\hline
  $19A-B$ & 180 &  $9$ & {\rm False} & {\rm True}& $252$& (100, 80, 0)\\\hline
\end{tabular}

$PSL(2,16)$, order 4080\hfill\\
\begin{tabular}{|c|c|c|c|c|c|c|}
  \hline
  $\CC$ & $|\CC|$ & $\chi_\CC(\CC)$ & Real & Irred & $\lambda_{\text{max}}$ & Signature  \\
  \hline
  $2A$  & 255   &$15$ & {\rm True} & {\rm False} (17) & $465$ & (255, 0, 0) \\\hline
  $3A$  & 272   &$2$ &  {\rm True} & {\rm True}      & $514$ & (136, 136, 0) \\\hline
  $5A-B$  & 272 &$2$ &  {\rm True} & {\rm True}      & $514$ & (136, 136, 0) \\\hline
  $15A-D$ & 272 & $2$ & {\rm True} & {\rm True}       & $514$ & (136, 136, 0) \\\hline
  $17A-H$ & 240 & $2$ & {\rm True} & {\rm True}       & $480$ & (120, 120, 0) \\\hline
\end{tabular}

$PSL(3,3)$, order 5616\hfill\\
\begin{tabular}{|c|c|c|c|c|c|c|}
  \hline
  $\CC$ & $|\CC|$ & $\chi_\CC(\CC)$ & Real & Irred & $\lambda_{\text{max}}$ & Signature \\
  \hline
  $2A$  & 117   & $13$ & {\rm True}  & {\rm True}   & $489$ & (117, 0, 0)\\\hline
  $3A$  & 104   & $14$ & {\rm True}  & {\rm True}   & $412$ & (52, 52, 0)\\\hline
  $3B$  & 624   & $6$ &  {\rm True}  & {\rm True}  & $1224$ & (312, 312, 0)\\\hline
  $4A$  & 702   & $2$ &  {\rm True}  & {\rm True}  & $1356$ & (351, 351, 0)\\\hline
  $6A$  & 936   & $2$ &  {\rm True}  & {\rm True}  & $1848$ & (468, 468, 0)\\\hline
  $8A-B$  & 702 & $2$ &  {\rm False} & {\rm True}  & $600$ & (337, 365, 0)\\\hline
  $13A-B$ & 432 & $3$ &  {\rm False} & {\rm True}  & $399$ & (224, 208, 0)\\\hline
  $13C-D$ & 432 & $3$ &  {\rm False} & {\rm True}  & $399$ & (236, 196, 0)\\\hline
\end{tabular}

$PSU(3,3)\cong {G_2^2}' $, order 6048\hfill\\
\begin{tabular}{|c|c|c|c|c|c|c|}
  \hline
  $\CC$ & $|\CC|$ & $\chi_\CC(\CC)$ & Real & Irred & $\lambda_{\text{max}}$ & Signature \\
  \hline
  $2A$  & 63    & $7$ &  {\rm True}  & {\rm True}   & $177$ & (63, 0, 0)\\\hline
  $3A$  & 56    & $2$ &  {\rm True}  & {\rm True}   & $112$ & (28, 28, 0)\\\hline
  $3B$  & 672   & $6$ &  {\rm True}  & {\rm True}   & $1332$& (336, 336, 0)\\\hline
  $4A-B$  & 63  & $7$ &  {\rm False} & {\rm True}   & $105$ & (22, 14, 27) \\\hline
  $4C$  & 378   & $6$ &  {\rm True}  & {\rm True}   & $852$ & (189, 189, 0) \\\hline
  $6A$  & 504   & $2$ &  {\rm True}  & {\rm True}   & $1104$& (252, 252, 0) \\\hline
  $7A-B$  & 864 & $3$ &  {\rm False} & {\rm True}   & $555$ & (436, 428, 0) \\\hline
  $8A-B$  & 756 & $2$ &  {\rm False} & {\rm True}   & $752$ & (364, 365, 27) \\\hline
  $12A-B$ & 504 & $2$ &  {\rm False} & {\rm True}   & $480$ & (238, 224, 42) \\\hline
\end{tabular}

$PSL(2,23)$, order 6072\hfill\\
\begin{tabular}{|c|c|c|c|c|c|c|}
  \hline
  $\CC$ & $|\CC|$ & $\chi_\CC(\CC)$ & Real & Irred & $\lambda_{\text{max}}$ & Signature \\
  \hline
  $2A$  & 253    &$13$ &  {\rm True}  & {\rm True}  & $529$ & (253, 0, 0)\\\hline
  $3A$  & 506    &$2$  &  {\rm True}  & {\rm True}  & $988$ & (253, 253, 0)\\\hline
  $4A$  & 506    &$2$  &  {\rm True}  & {\rm True}  & $988$ & (253, 253, 0)\\\hline
  $6A$  & 506    &$2$  &  {\rm True}  & {\rm True}  & $988$ & (253, 253, 0)\\\hline
  $11A-E$ & 552  & $2$ &  {\rm True}  & {\rm True}& $1038$ & (276, 276, 0)\\\hline
  $12A-B$ & 506  & $2$ &  {\rm True}  & {\rm True}& $988$ & (253, 253, 0)\\\hline
  $23A-B$ & 264  & $11$&  {\rm False} & {\rm True}& $374$ & (144, 120, 0)\\\hline
\end{tabular}

$PSL(2,25)$, order 7800\hfill\\
\begin{tabular}{|c|c|c|c|c|c|c|}
  \hline
  $\CC$ & $|\CC|$ & $\chi_\CC(\CC)$ & Real & Irred & $\lambda_{\text{max}}$ & Signature \\
  \hline
  $2A$  & 325   &$13$ & {\rm True} & {\rm True}& $601$  & (325, 0, 0) \\\hline
  $3A$  & 650   &$2$  & {\rm True} & {\rm True}& $1228$  & (325, 325, 0) \\\hline
  $4A$  & 650   &$2$  & {\rm True} & {\rm True}& $1228$  & (325, 325, 0) \\\hline
  $5A-B$  & 312 &$12$ & {\rm True} & {\rm True}& $744$  & (156, 156, 0) \\\hline
  $6A$  & 650   &$2$  & {\rm True} & {\rm True}& $1228$  & (325, 325, 0) \\\hline
  $12A-B$ & 650 & $2$ & {\rm True} & {\rm True}& $1228$ & (325, 325, 0) \\\hline
  $13A-F$ & 600 & $2$ & {\rm True} & {\rm True}& $1174$ & (300, 300, 0) \\\hline
\end{tabular}

$M_{11}$, order 7920\hfill\\
\begin{tabular}{|c|c|c|c|c|c|c|}
  \hline
  $\CC$ & $|\CC|$ & $\chi_\CC(\CC)$ & Real & Irred & $\lambda_{\text{max}}$ & Signature \\
  \hline
  $2A$  & 165    &$13$ &  {\rm True}& {\rm True}  & $489$& (165, 0, 0)\\\hline 
  $3A$  & 440    &$8$  &  {\rm True}& {\rm True}  & $946$& (220, 220, 0)\\\hline
  $4A$  & 990    &$2$  &  {\rm True}& {\rm True}  & $2108$& (495, 495, 0)\\\hline
  $5A$  & 1584   & $4$ &  {\rm True}& {\rm True}  & $3096$& (792, 792, 0)\\\hline
  $6A$  & 1320   & $2$ &  {\rm True}& {\rm True}  & $2568$& (660, 660, 0)\\\hline
  $8A-B$  & 990  &$2$  &  {\rm False}& {\rm True}& $920$& (515, 475, 0)\\\hline
  $11A-B$ & 720  & $5$ &  {\rm False}& {\rm True}& $575$& (355, 365, 0)\\\hline
\end{tabular}

$PSL(2,27)$, order 9828\hfill\\
\begin{tabular}{|c|c|c|c|c|c|c|}
  \hline
  $\CC$ & $|\CC|$ & $\chi_\CC(\CC)$ & Real & Irred & $\lambda_{\text{max}}$ & Signature \\
  \hline
  $2A$  & 351   & $15$ &{\rm True}  & {\rm True}  & $729$& (351, 0, 0)\\\hline
  $3A-B$  & 364 & $13$ &{\rm False} & {\rm True}& $520$& (196, 168, 0)\\\hline
  $7A-C$  & 702 & $2$  &{\rm True}  & {\rm True}& $1376$& (351, 351, 0)\\\hline
  $13A-F$ & 756 &  $2$ &{\rm True}  & {\rm True}& $1434$ & (378, 378, 0)\\\hline
  $14A-C$ & 702 &  $2$ &{\rm True}  & {\rm True}& $1376$ & (351, 351, 0)\\\hline
\end{tabular}

$PSL(2,29)$, order 12180\hfill\\
\begin{tabular}{|c|c|c|c|c|c|c|}
  \hline
  $\CC$ & $|\CC|$ & $\chi_\CC(\CC)$ & Real & Irred & $\lambda_{\text{max}}$ & Signature \\
  \hline
  $2A$  & 435   &  $15$ & {\rm True} & {\rm True}  & $813$& (435, 0, 0)\\\hline
  $3A$  & 812   &  $2$  & {\rm True} & {\rm True}  & $1594$& (406, 406, 0)\\\hline
  $5A-B$  & 812 &  $2$  & {\rm True} & {\rm True}& $1594$& (406, 406, 0)\\\hline
  $7A-C$  & 870 &  $2$  & {\rm True} & {\rm True}& $1656$& (435, 435, 0)\\\hline
  $14A-C$ & 870 &  $2$ & {\rm True} & {\rm True}& $1656$ & (435, 435, 0)\\\hline
  $15A-D$ & 812 &  $2$ & {\rm True} & {\rm True}& $1594$ & (406, 406, 0)\\\hline
  $29A-B$ & 420 &  $14$ &{\rm True} & {\rm True} & $1008$ & (210, 210, 0)\\\hline
\end{tabular}

$PSL(2,31)$, order 14880\hfill\\
\begin{tabular}{|c|c|c|c|c|c|c|}
  \hline
  $\CC$ & $|\CC|$ & $\chi_\CC(\CC)$ & Real & Irred & $\lambda_{\text{max}}$ & Signature \\
  \hline
  $2A$  & 465   &$17$  &  {\rm True} & {\rm True}  & $961$& (465, 0, 0)\\\hline
  $3A$  & 992   &$2$   &  {\rm True} & {\rm True}  & $1894$& (496, 496, 0)\\\hline
  $4A$  & 930   &$2$   &  {\rm True} & {\rm True}  & $1828$& (465, 465, 0)\\\hline
  $5A-B$  & 992 &$2$   &  {\rm True} & {\rm True}& $1894$& (496, 496, 0)\\\hline
  $8A$  & 930   &$2$   &  {\rm True} & {\rm True}  & $1828$& (465, 465, 0)\\\hline
  $15A-D$ & 992 & $2$  &  {\rm True} & {\rm True}& $1894$ & (496, 496, 0)\\\hline
  $16A-E$ & 930 & $2$  &  {\rm True} & {\rm True}& $1828$ & (465, 465, 0)\\\hline
  $31A-B$ & 480 & $15$ &  {\rm False}& {\rm True}& $690$ & (256, 224, 0)\\\hline
\end{tabular}

$A_8$, order 20160\hfill\\
\begin{tabular}{|c|c|c|c|c|c|c|}
  \hline
  $\CC$ & $|\CC|$ & $\chi_\CC(\CC)$ & Real & Irred & $\lambda_{\text{max}}$ & Signature \\
  \hline
  $2A$  & 105  &  $25$ & {\rm True}  &  {\rm True}  & $849$& (105, 0, 0) \\\hline
  $2B$  & 210  &  $18$ & {\rm True}  &  {\rm True}  & $996$& (210, 0, 0) \\\hline
  $3A$  & 112  &  $22$ & {\rm True}  &  {\rm True}  & $784$& (56, 56, 0) \\\hline
  $3B$  & 1120 &   $4$ & {\rm True}  &  {\rm True}  & $3028$& (560, 560, 0) \\\hline
  $4A$  & 1260 &   $8$ & {\rm True}  &  {\rm True}  & $3280$& (630, 630, 0) \\\hline
  $4B$  & 2520 &   $4$ & {\rm True}  &  {\rm True}  & $4736$& (1260, 1260, 0) \\\hline
  $5A$  & 1344 &   $4$ & {\rm True}  &  {\rm True}  & $2996$& (672, 672, 0) \\\hline
  $6A$  & 1680 &   $6$ & {\rm True}  &  {\rm True}  & $3600$& (840, 840, 0) \\\hline
  $6B$  & 3360 &   $2$ & {\rm True}  &  {\rm True}  & $6168$& (1680, 1680, 0) \\\hline
  $7A-B$  & 2880 & $3$ & {\rm False} &  {\rm True}& $2466$& (1375, 1505, 0)  \\\hline
  $15A-B$ & 1344 & $4$ & {\rm False} &  {\rm True} & $1556$ & (597, 747, 0)  \\\hline
\end{tabular}

$PSL(3,4)$, order 20160\hfill\\
\begin{tabular}{|c|c|c|c|c|c|c|}
  \hline
  $\CC$ & $|\CC|$ & $\chi_\CC(\CC)$ & Real & Irred & $\lambda_{\text{max}}$ & Signature \\
  \hline
  $2A$ & 315    & $27$  & {\rm True}  & {\rm True}  & $1305$ & (315, 0, 0)\\ \hline
  $3A$ & 2240   &  $8$  & {\rm True}  & {\rm True}  & $4888$ & (1120, 1120, 0)\\ \hline
  $4A-C$ & 1260 &  $12$ & {\rm True} & {\rm True}  & $3312$ & (630, 630, 0)\\\hline
  $5A-B$ & 4032 &  $2$  & {\rm True} & {\rm True}  & $7284$ & (2016, 2016, 0)\\\hline
  $7A-B$ & 2880 &  $3$  & {\rm False} & {\rm True}  & $2466$ & (1398, 1302, 180) \\ \hline
\end{tabular}

$PSL(2,37)$, order 25308\hfill\\
\begin{tabular}{|c|c|c|c|c|c|c|}
  \hline
  $\CC$ & $|\CC|$ & $\chi_\CC(\CC)$ & Real & Irred & $\lambda_{\text{max}}$ & Signature \\
  \hline
  $2A$  & 703  &  $19$ &  {\rm True} & {\rm True} & $1333$ & (703, 0, 0)\\\hline
  $3A$  & 1406 &   $2$ &  {\rm True} & {\rm True} & $2704$ & (703, 703, 0)\\\hline
  $6A$  & 1406 &   $2$ &  {\rm True} & {\rm True} & $2704$ & (703, 703, 0)\\\hline
  $9A-C$  & 1406 & $2$ &  {\rm True} & {\rm True} & $2704$ & (703, 703, 0)\\\hline
  $18A-C$ & 1406 & $2$ &  {\rm True} & {\rm True} & $2704$ & (703, 703, 0)\\\hline
  $19A-I$ & 1332 & $2$ &  {\rm True} & {\rm True} & $2626$ & (666, 666, 0)\\\hline
  $37A-B$ & 684 & $18$ &  {\rm True} & {\rm True} & $1656$ & (342, 342, 0)\\\hline
\end{tabular}

$PSU(4,2)$, order 25920\hfill\\
\begin{tabular}{|c|c|c|c|c|c|c|}
  \hline
  $\CC$ & $|\CC|$ & $\chi_\CC(\CC)$ & Real & Irred & $\lambda_{\text{max}}$ & Signature \\
  \hline
  $2A$  & 45     &$13$ & {\rm True}   & {\rm True}  & $201$ & (45, 0, 0) \\\hline
  $2B$  & 270    &$22$ & {\rm True}   & {\rm True}  & $1188$ & (270, 0, 0) \\\hline
  $3A-B$  & 40   &$13$ &  {\rm False} & {\rm True} & $196$ & (40, 0, 0) \\\hline
  $3C$  & 240    &$6$  & {\rm True}   & {\rm True}  & $720$ & (120, 120, 0) \\\hline
  $3D$  & 480    &$12$ & {\rm True}   & {\rm True}  & $1548$ & (240, 240, 0) \\\hline
  $4A$  & 540    &$8$  & {\rm True}   & {\rm True}  & $1488$ & (270, 270, 0) \\\hline
  $4B$  & 3240   &$4$ &  {\rm True}   & {\rm True}  & $5440$ & (1620, 1620, 0) \\\hline
  $5A$  & 5184   &$4$ &  {\rm True}   & {\rm True}  & $9836$ & (2592, 2592, 0) \\\hline
  $6A-B$  & 360  &$5$  & {\rm False} & {\rm True} & $708$ & (231, 129, 0) \\\hline
  $6C-D$  & 720  &$4$  & {\rm False} & {\rm True} & $1272$ & (364, 356, 0) \\\hline
  $6E$  & 1440   &$2$ &  {\rm True}   & {\rm True}  & $3336$ & (720, 720, 0) \\\hline
  $6F$  & 2160   &$2$ &  {\rm True}   & {\rm True}  & $4176$ & (1080, 1080, 0) \\\hline
  $9A-B$  & 2880 &$3$ &  {\rm False} & {\rm True} & $2646$ & (1595, 1285, 0) \\\hline
  $12A-B$ & 2160 & $2$&  {\rm False} & {\rm True} & $1824$ & (1035, 1125, 0) \\\hline
\end{tabular}

$Suz_8$, order 29120\hfill\\
\begin{tabular}{|c|c|c|c|c|c|c|}
  \hline
  $\CC$ & $|\CC|$ & $\chi_\CC(\CC)$ & Real & Irred & $\lambda_{\text{max}}$ & Signature \\
  \hline
  $2A$  & 455    & $7$ & {\rm True}    &{\rm False} (65)& $497$& (455, 0, 0)\\\hline
  $4A-B$  & 1820 & $4$ & {\rm False} &{\rm True}      & $2768$ & (755, 1065, 0)\\\hline
  $5A$  & 5824   & $4$ & {\rm True}    &{\rm True}      & $9796$ & (2912, 2912, 0)\\\hline
  $7A-C$  & 4160 & $2$ & {\rm True}  &{\rm True}      & $7690$ & (2080, 2080, 0)\\\hline
  $13A-C$ & 2240 & $4$ & {\rm True}  &{\rm True}       & $4748$ & (1120, 1120, 0)\\\hline
\end{tabular}

$PSL(2,32)$, order 32736\hfill\\
\begin{tabular}{|c|c|c|c|c|c|c|}
  \hline
  $\CC$ & $|\CC|$ & $\chi_\CC(\CC)$ & Real & Irred & $\lambda_{\text{max}}$ & Signature \\
  \hline
  $2A$  & 1023   &$31$ &  {\rm True} & {\rm False} (33)& $1953$ & (1023, 0, 0)\\\hline
  $3A$  & 992    &$2$  &  {\rm True} & {\rm True}      & $1984$ & (496, 496, 0)\\\hline
  $11A-E$ & 992  &$2$  &  {\rm True} & {\rm True}    & $1984$ & (496, 496, 0)\\\hline
  $31A-O$ & 1056 &$2$  &  {\rm True} & {\rm True}    & $2050$ & (528, 528, 0)\\\hline
  $33A-J$ & 992  &$2$  &  {\rm True} & {\rm True}    & $1984$ & (496, 496, 0)\\\hline
\end{tabular}

$PSL(2,41)$, order 34440\hfill\\
\begin{tabular}{|c|c|c|c|c|c|c|}
  \hline
  $\CC$ & $|\CC|$ & $\chi_\CC(\CC)$ & Real & Irred & $\lambda_{\text{max}}$ & Signature \\
  \hline
  $2A$  & 861    & $21$ & {\rm True} &  {\rm True}   & $1641$& (861, 0, 0)\\\hline
  $3A$  & 1640   & $2$  &{\rm True} &  {\rm True}   & $3238$& (820, 820, 0)\\\hline
  $4A$  & 1722   & $2$  &{\rm True} &  {\rm True}   & $3324$& (861, 861, 0)\\\hline
  $5A-B$  & 1722 & $2$  &{\rm True}  & {\rm True} & $3324$& (861, 861, 0)\\\hline
  $7A-C$  & 1640 & $2$  &{\rm True}  & {\rm True} & $3238$& (820, 820, 0)\\\hline
  $10A-B$ & 1722 & $2$  & {\rm True}  & {\rm True} & $3324$& (861, 861, 0)\\\hline
  $20A-D$ & 1722 & $2$  & {\rm True}  & {\rm True} & $3324$& (861, 861, 0)\\\hline
  $21A-F$ & 1640 & $2$  & {\rm True}  & {\rm True} & $3238$& (820, 820, 0)\\\hline
  $41A-B$ & 840  & $20$ &{\rm True}  & {\rm True} & $2040$& (420, 420, 0)\\\hline
\end{tabular}

$PSL(2,43)$, order 39732\hfill\\
\begin{tabular}{|c|c|c|c|c|c|c|}
  \hline
  $\CC$ & $|\CC|$ & $\chi_\CC(\CC)$ & Real & Irred & $\lambda_{\text{max}}$ & Signature \\
  \hline
  $2A$  & 903    &$23$  & {\rm True} & {\rm True}    & $1849$& (903, 0, 0)\\\hline
  $3A$  & 1892   & $2$  & {\rm True} & {\rm True}    & $3658$& (946, 946, 0)\\\hline
  $7A-C$  & 1892 & $2$  & {\rm True} & {\rm True}  & $3658$& (946, 946, 0)\\\hline
  $11A-E$ & 1806 & $2$ &  {\rm True} & {\rm True}  & $3568$& (903, 903, 0)\\\hline
  $21A-F$ & 1892 & $2$ &  {\rm True} & {\rm True}  & $3658$& (946, 946, 0)\\\hline
  $22A-E$ & 1806 & $2$ &  {\rm True} & {\rm True}  & $3568$& (903, 903, 0)\\\hline
  $43A-B$ & 924  & $21$ & {\rm False}& {\rm True}  & $1344$& (484, 440, 0)\\\hline
\end{tabular}

$PSL(2,47)$, order 51888\hfill\\
\begin{tabular}{|c|c|c|c|c|c|c|}
  \hline
  $\CC$ & $|\CC|$ & $\chi_\CC(\CC)$ & Real & Irred & $\lambda_{\text{max}}$ & Signature \\
  \hline
  $2A$  & 1081   &$25$ &  {\rm True} & {\rm True}   & $2209$& (1081, 0, 0)\\\hline
  $3A$  & 2162   &$2$  &  {\rm True} & {\rm True}   & $4276$& (1081, 1081, 0)\\\hline
  $4A$  & 2162   &$2$  &  {\rm True} & {\rm True}   & $4276$& (1081, 1081, 0)\\\hline
  $6A$  & 2162   &$2$  &  {\rm True} & {\rm True}   & $4276$& (1081, 1081, 0)\\\hline
  $8A-B$  & 2162 &$2$  &  {\rm True} & {\rm True} & $4276$& (1081, 1081, 0)\\\hline
  $12A-B$ & 2162 & $2$ &  {\rm True} & {\rm True} & $4276$& (1081, 1081, 0)\\\hline
  $23A-K$ & 2256 & $2$ &  {\rm True} & {\rm True} & $4374$& (1128, 1128, 0)\\\hline
  $24A-D$ & 2162 & $2$ &  {\rm True} & {\rm True} & $4276$& (1081, 1081, 0)\\\hline
  $47A-B$ & 1104 & $23$ & {\rm False}& {\rm True}  & $1610$& (576, 528, 0)\\\hline
\end{tabular}

$PSL(2,49)$, order 58800\hfill\\
\begin{tabular}{|c|c|c|c|c|c|c|}
  \hline
  $\CC$ & $|\CC|$ & $\chi_\CC(\CC)$ & Real & Irred & $\lambda_{\text{max}}$ & Signature  \\
  \hline
  $2A$  & 1225   & $25$ & {\rm True} & {\rm True}  & $2353$& (1225, 0, 0)\\\hline
  $3A$  & 2450   & $2$  & {\rm True} & {\rm True}  & $4756$& (1225, 1225, 0)\\\hline
  $4A$  & 2450   & $2$  & {\rm True} & {\rm True}  & $4756$& (1225, 1225, 0)\\\hline
  $5A-B$  & 2352 & $2$  & {\rm True} & {\rm True}& $4654$& (1176, 1176, 0)\\\hline
  $6A$  & 2450   & $2$  & {\rm True} & {\rm True}  & $4756$& (1225, 1225, 0)\\\hline
  $7A-B$  & 1200 & $24$ & {\rm True} & {\rm True}& $2928$& (600, 600, 0)\\\hline
  $8A-B$  & 2450 & $2$  & {\rm True} & {\rm True}& $4756$& (1225, 1225, 0)\\\hline
  $12A-B$ & 2450 &  $2$ & {\rm True} & {\rm True}& $4756$ & (1225, 1225, 0)\\\hline
  $24A-D$ & 2450 &  $2$ & {\rm True} & {\rm True}& $4756$ & (1225, 1225, 0)\\\hline
  $25A-J$ & 2352 &  $2$ & {\rm True} & {\rm True}& $4654$ & (1176, 1176, 0)\\\hline
\end{tabular}

$PSU(3,4)$, order 62400\hfill\\
\begin{tabular}{|c|c|c|c|c|c|c|}
  \hline
  $\CC$ & $|\CC|$ & $\chi_\CC(\CC)$ & Real & Irred & $\lambda_{\text{max}}$ & Signature \\
  \hline
  $2A$ & $195$     & $3$  & {\rm True} & {\rm False} (65)& $201$& (195, 0, 0)\\\hline
  $3A$ & $4160$    & $2$  & {\rm True} & {\rm True}     & $8134$& (2080, 2080, 0)\\\hline
  $4A$ & $3900$    & $12$ & {\rm True} & {\rm True}     & $7824$& (1950, 1950, 0)\\\hline
  $5A-D$ & $208$   & $13$ & {\rm False} & {\rm True}   & $484$& (79, 129, 0)\\\hline
  $5E-F$ & $2496$  & $6$  & {\rm True} & {\rm True}   & $5436$& (1248, 1248, 0)\\\hline
  $10A-D$ & $3120$ & $3$  & {\rm False} & {\rm True} & $3756$& (1586, 1534, 0)\\\hline
  $13A-D$ & $4800$ & $3$  & {\rm False} & {\rm True} & $3948$& (2310, 2490, 0)\\\hline
  $15A-D$ & $4160$ & $2$  & {\rm False} & {\rm True} & $4054$& (2041, 2119, 0)\\\hline
\end{tabular}

$PSL(2,53)$, order 74412\hfill\\
\begin{tabular}{|c|c|c|c|c|c|c|}
  \hline
  $\CC$ & $|\CC|$ & $\chi_\CC(\CC)$ & Real & Irred & $\lambda_{\text{max}}$ & Signature \\
  \hline
  $2A$ & $1431$     & $27$ & {\rm True} & {\rm True}    & $2757$& (1431, 0, 0)\\\hline
  $3A$ & $2756$     & $2$  & {\rm True} & {\rm True}    & $5458$& (1378, 1378, 0)\\\hline
  $9A-C$ & $2756$  & $2$  & {\rm True} & {\rm True} & $5458$& (1378, 1378, 0)\\\hline
  $13A-F$ & $2862$ & $2$  & {\rm True}& {\rm True} & $5568$& (1431, 1431, 0)\\\hline
  $26A-F$ & $2862$ & $2$  & {\rm True}& {\rm True} & $5568$& (1431, 1431, 0)\\\hline
  $27A-I$ & $2756$ & $2$  & {\rm True}& {\rm True} & $5458$& (1378, 1378, 0)\\\hline
  $53A-B$ & $1404$ & $26$ & {\rm True}& {\rm True} & $3432$& (702, 702, 0)\\\hline
\end{tabular}

\end{document}